\documentclass[11pt,a4paper,reqno]{amsart}
\usepackage[czech,english]{babel}
\usepackage{amscd,amssymb}
\usepackage{dsfont}
\usepackage{url}
\usepackage[all]{xy}
\usepackage{hyperref}
\title[Covering classes and $1$-tilting cotorsion pairs]{Covering classes and $1$-tilting cotorsion pairs over commutative rings}
\author[S.~Bazzoni]{Silvana Bazzoni}
\address[Silvana Bazzoni]{%
Dipartimento di Matematica ``Tullio Levi-Civita'' \\
Universit\`a di Padova \\
Via Trieste 63, 35121 Padova (Italy)}

\email{bazzoni@math.unipd.it}
\author[G.~Le Gros]{Giovanna Le Gros}
\address[Giovanna Le Gros]{%
Dipartimento di Matematica ``Tullio Levi-Civita'' \\
Universit\`a di Padova \\
Via Trieste 63, 35121 Padova (Italy)}

\email{giovannagiulia.legros@math.unipd.it}

\thanks{Research supported by grants from Ministero dell'Istruzione, dell'Universit\`a e della Ricerca (PRIN: ``Categories, Algebras: Ring-Theoretical and Homological Approaches (CARTHA)'') and Dipartimento di Matematica ``Tullio Levi-Civita'' of the Universit\`a di Padova (Research program DOR1828909 ``Anelli e categorie di moduli''). }

\newcommand{\sto}{{\v{S}}{\v{t}}ov{\'{\i}}{\v{c}}ek}

\newcommand{\bbN}{\mathbb{N}}
\newcommand{\bbZ}{\mathbb{Z}}
\newcommand{\bbQ}{\mathbb{Q}}

\newcommand{\m}{\mathfrak{m}}
\newcommand{\n}{\mathfrak{n}}
\newcommand{\p}{\mathfrak{p}}
\newcommand{\q}{\mathfrak{q}}

\newcommand{\Max}{\operatorname{Max}}    

\newcommand{\Hom}{\operatorname{Hom}}

\newcommand{\End}{\operatorname{End}}

\newcommand{\Ext}{\operatorname{Ext}}
\newcommand{\Tor}{\operatorname{Tor}}

\newcommand{\Gen}{\operatorname{Gen}}
\newcommand{\Ker}{\operatorname{Ker}}

\newcommand{\Tr}{\operatorname{Tr}}
\newcommand{\Img}{\operatorname{Im}}
\newcommand{\Coker}{\operatorname{Coker}}

\newcommand{\wdim}{\operatorname{w.dim}}
\newcommand{\pdim}{\operatorname{p.dim}}
\newcommand{\idim}{\operatorname{inj.dim}}

\newcommand{\Fdim}{\operatorname{F.dim}}

\newcommand{\fdim}{\operatorname{f.dim}}

\newcommand{\id}{\operatorname{id}}
\newcommand{\soc}{\operatorname{soc}}

\newcommand{\adjK}{\big( (- \otimes_RK), \Hom_R(K,-) \big)}

\newcommand{\A}{\mathcal{A}}
\newcommand{\B}{\mathcal{B}}
\newcommand{\C}{\mathcal{C}}
\newcommand{\D}{\mathcal{D}}
\newcommand{\E}{\mathcal{E}}
\newcommand{\F}{\mathcal{F}}
\newcommand{\G}{\mathcal{G}}
\newcommand{\clH}{\mathcal{H}}
\newcommand{\I}{\mathcal{I}}

\newcommand{\clP}{\mathcal{P}}

\newcommand{\clS}{\mathcal{S}}

\newcommand{\T}{\mathcal{T}}

\newcommand{\Modr}[1]{\mathrm{Mod}\textrm{-}{#1}}

\newcommand{\modr}[1]{\mathrm{mod}\textrm{-}{#1}}

\newcommand{\ModR}{\mathrm{Mod}\textrm{-}R}
\newcommand{\RMod}{R\textrm{-}\mathrm{Mod}}

\newcommand{\ModU}{\mathrm{Mod}\textrm{-}U}

\newcommand{\Mod}{\mathrm{Mod}\textrm{-}}

\newcommand{\modR}{\mathrm{mod}\textrm{-}R}

\newcommand{\ucontra}{u\textrm{-}\mathbf{contra}}

\newcommand{\Add}{\mathrm{Add}}
\newcommand{\Ann}{\mathrm{Ann}}

\theoremstyle{plain}
\newtheorem{thm}{Theorem}[section]
\newtheorem{lem}[thm]{Lemma}
\newtheorem{prop}[thm]{Proposition}
\newtheorem{cor}[thm]{Corollary}

\newtheorem{rem}[thm]{Remark}

\newtheorem{facts}[thm]{Facts}

\theoremstyle{definition}
\newtheorem{defn}[thm]{Definition}

\newtheorem{set}[thm]{Setting}

\newtheorem{expl}[thm]{Example}

\theoremstyle{remark}

\begin{document}
 \maketitle
\begin{abstract}

We are interested in characterising the commutative rings for which a $1$-tilting cotorsion pair $(\A, \T)$ provides for covers, that is when the class $\A$ is a covering class. We use Hrbek's bijective correspondence between the $1$-tilting cotorsion pairs over a commutative ring $R$ and the faithful finitely generated Gabriel topologies on $R$. Moreover, we use results of Bazzoni-Positselski, in particular a generalisation of Matlis equivalence and their characterisation of covering classes for $1$-tilting cotorsion pairs arising from flat injective ring epimorphisms. Explicitly, if $\G$ is the Gabriel topology associated to the $1$-tilting cotorsion pair $(\A, \T)$, and $R_\G$ is the ring of quotients with respect to $\G$, we show that if $\A$ is covering then $\G$ is a perfect localisation (in Stenstr{\"o}m's sense \cite{Ste75}) and the localisation $R_\G$ has projective dimension at most one. Moreover, we show that $\A$ is covering if and only if both the localisation $R_\G$ and the quotient rings $R/J$ are perfect rings for every $J \in \G$. Rings satisfying the latter two conditions are called $\G$-almost perfect.

\end{abstract}

\section{Introduction}
The study of the existence of approximations for a certain class $\C$ is a method for studying the category of all modules. An objective
is to characterise the rings over which every module has a cover (right minimal approximation) provided by $\C$ and furthermore to characterise the class $\C$ itself. 

 In this paper we study when for a $1$-tilting cotorsion pair $(\A, \T)$ over a commutative ring $R$ the class $\A$ provides for covers. 

One of the first examples of the power of studying the module category using the existence of minimal approximations was done for the class of projective modules by Bass \cite{Bass}. Bass introduced and characterised the class of perfect rings which are exactly the rings over which every module has a projective cover. Moreover, he showed that this is equivalent to the class of projective modules being closed under direct limits. It is interesting to study this closure property for the general case of covering classes.

In fact, a famous theorem of Enochs says that if a class $\C$ in $\ModR$ is closed under direct limits, then any module that has a $\C$-precover has a $\C$-cover \cite{Eno}. The converse problem, that is if a covering class $\mathcal{C}$ is necessarily closed under direct limits, is still an open problem which is known as Enochs Conjecture.

 Some significant advancements have been made towards Enochs Conjecture in recent years.
In 2017, Angeleri H\"ugel-\v Saroch-Trlifaj in \cite{AST17} proved that Enochs Conjecture holds for a large class of cotorsion pairs. Explicity, they proved that for a cotorsion pair $(\mathcal{A}, \mathcal{B})$ such that $\mathcal{B}$ is closed under direct limits, $\mathcal{A}$ is covering if and only if it is closed under direct limits. In particular, this holds for all tilting cotorsion pairs. The result of Angeleri H\"ugel-\v Saroch-Trlifaj is based on methods developed in \v{S}aroch's paper \cite{S17}, which uses sophisticated set-theoretical methods in homological algebra. 

Moreover, in a very recent paper by Bazzoni-Positselski-\sto~in 2020 \cite{BPS}, it is proved using an algebraic approach that for a cotorsion pair $(\A, \B)$ such that $\B$ is closed under direct limits and a module $M \in \A \cap \B$, if $\Add (M )$ is covering then $\Add (M)$ is closed under direct limits. In particular, it follows that Enochs Conjecture holds for tilting cotorsion pairs.

We are interested in giving ring theoretic characterisations of the commutative rings for which the class $\A$ of a $1$-tilting cotorsion pair provides for covers, by purely algebraic methods. As mentioned in the previous paragraphs, Enochs conjecture is known to hold for these cotorsion pairs, although via our characterisation of such rings we find yet another proof of Enochs Conjecture. 

In our study we use extensively the bijective correspondence between $1$-tilting cotorsion pairs over commutative rings and faithful finitely generated Gabriel topologies as demonstrated by Hrbek in \cite{H}. 
 More precisely, Hrbek associates to a $1$-tilting class $\T$ the collection of ideals $J\leq R$ which ``divide'' $\T$, that is $\{J \mid JT = T, \forall T \in \T \}$, and in the converse direction he associates to a faithful finitely generated Gabriel topology $\G$ the $1$-tilting class, denoted $\D_\G$, of the $\G$-divisible modules, that is the modules $M$ such that $JM=M$ for every $J\in \G$.

 We are interested in a particular type of Gabriel topology. The ring of quotients of $R$ with respect to $\G$, denoted $R_\G$, is $\G$-divisible if and only if $\G$ arises from a perfect localisation, that is $\psi_R\colon R \to R_\G$ is a flat ring epimorphism and $\G = \{J \leq R \mid \psi_R(J)R_\G = R_\G \}$. Following Stenstr{\"o}m's terminology ~\cite{Ste75}, these Gabriel topologies are called perfect Gabriel topologies. 

In Section \ref{S:cov-Rg-div} we first prove that if the class $\A$ in a $1$-tilting cotorsion pair $(\A, \D_\G)$ over a commutative ring is covering, then $\G$ arises from a perfect localisation and that the projective dimension of $R_\G$ is at most one (see Lemma~\ref{L:cov-G-perf} and Proposition~\ref{P:R_G-tilting})
%
Then, by work of Angeleri H\"{u}gel-S\'{a}nchez in \cite{AS}, $R_\G \oplus R_\G/R$ is a $1$-tilting module with associated cotorsion pair $(\A, \D_\G)$. 

In this situation, we have a much larger range of theories to use, in particular the works of Bazzoni-Positselski and Positselski.

Indeed, in \cite{BP4}, using the theory of contramodules and the tilting-cotilting correspondence from \cite{PS19}, Bazzoni-Positselski give classification results for some $1$-tilting cotorsion pairs satisfying Enochs Conjecture---those which arise from an injective homological ring epimorphism $u\colon R \to U$ in the sense of \cite{AS}. 

 In \cite{BLG}, we prove that if a $1$-tilting class over a commutative ring $R$ is enveloping, then the tilting module arises from an injective flat epimorphism and gave a ring theoretic characterisation of the ring in terms of perfectness of the quotient rings $R/J$ for every ideal $J$ in the associated Gabriel topology.
 
 This paper concerns the covering side.

We briefly summarise the results of this paper as well as some implications. Consider a $1$-tilting cotorsion pair $(\A, \T)$ over a commutative ring and the associated Gabriel topology $\G$. After proving that if $\A$ is covering then $\G$ is a perfect Gabriel topology and the projective dimension of $R_\G$ is at most one, we show the following characterisation (Theorem~\ref{T:characterisation-cov}):\begin{equation*}
 \A \text{ is covering} \Leftrightarrow \begin{cases}
 R_\G \text{ is a perfect ring} \\
 R/J \text{ is a perfect ring for each } J \in \G \\
 \end{cases}
\end{equation*}
It is interesting to note that if $R_\G$ is a perfect ring, then it follows that $\G$ is a perfect Gabriel topology (see Lemma~\ref{L:fdimRg-0}).

The above characterisations use \cite[Theorem 1.2]{BP3}, where Bazzoni and Positselski state that for a (not necessarily injective) ring epimorphism $u\colon R \to U$ such that $\Tor^R_1(U,U)=0$ and $K := U/u(R)$, there is a Matlis equivalence between the full subcategory of $u$-divisible $u$-comodules and the full subcategory of $u$-torsion-free $u$-contramodules via the adjunction pair $\big((- \otimes_R K), \Hom_R(K,-)\big).$ When $u\colon R \to U$ is a flat injective ring epimorphism of commutative rings as in our case, this becomes an equivalence between the $\G$-divisible $\G$-torsion modules and the $\G$-torsion-free $u$-contramodules.

This paper is structured as follows. We begin with some general preliminaries where we introduce minimal approximations and $1$-tilting classes in Section~\ref{S:prelim}.

Section~\ref{S:gab} makes up the background to our main results. We introduce Gabriel topologies as well as some more recent advancements and some of our own results.
 
Next in Section~\ref{S:cov-Rg-div}, we have an initial main result for a $1$-tilting cotorsion pair $(\A, \D_\G)$ over a commutative ring $R$ such that $\A$ is covering. We show that $\G$ is a perfect localisation and that $R_\G \oplus R_\G/R$ is a $1$-tilting module for the $1$-tilting cotorsion pair $(\A, \D_\G)$.

Next we introduce topological rings and $u$-contramodules for a ring epimorphism $u$ in Section~\ref{S:top}, which will be used for our main results. These results are collected from various papers of Positselski and Bazzoni-Positselski. 

In Section~\ref{S:cov-G-perf} we again consider a $1$-tilting cotorsion pair $(\A, \D_\G)$ over a commutative ring $R$ such that $\A$ is covering and show that $R$ is $\G$-almost perfect.

In Section~\ref{S:h-local} we introduce $\clH$-h-local rings with respect to a linear topology $\clH$ over a commutative ring, as a generalisation of results in \cite{BP1}.
In Section~\ref{S:Gperf} we show the converse of the combination of Section~\ref{S:cov-Rg-div} and ~\ref{S:cov-G-perf}. That is, if $(\A, \D_\G)$ is a $1$-tilting cotorsion pair over a commutative ring $R$ and $R$ is $\G$-almost perfect, we show that $\A$ is covering.

\subsection*{Acknowledgement}
The authors are grateful to Leonid Positselski for reading and providing helpful comments to an earlier version of this preprint. 

\section{Preliminaries}\label{S:prelim}
%
In this section we will recall some definitions and some notation.

All rings will be associative with a unit, $\ModR$ ($\RMod$) the category of right (left) $R$-modules over the ring $R$, and $\modR$ the full subcategory of $\ModR$ which is composed of all the modules which have a projective resolution consisting of only finitely generated projective modules.

 For a right $R$-module $M$ and a right ideal $I$ of $R$, we let $M[I]$ denote the submodule of $M$ of elements which are annihilated by the ideal $I$. That is, $M[I]:= \{x \in M \mid xI =0 \}$.
 
Let $\C $ be a class of right $R$-modules. The right $\Ext^1_R$-orthogonal and right $\Ext^\infty_R$-orthogonal classes of $\C$ are defined as follows. \[%
\C ^{\perp_1} =\{M\in \Modr R \ | \ \Ext_R^1(C,M)=0 \ {\rm for \
all\ } C\in \C\} 
\]
\[\C^\perp = \{M\in \Modr R \ | \ \Ext_R^i(C,M)=0 \ {\rm for \
all\ } C\in \C, \ {\rm for \
all\ } i\geq 1 \}\]
The left Ext-orthogonal classes ${}^{\perp_1} \C$ and ${}^\perp \C$ are defined symmetrically. 

If the class $\C$ has only one element, say $\C = \{X\}$, we write $X^{\perp_1}$ instead of $\{X\}^{\perp_1}$, and similarly for the other $\Ext$-orthogonal classes.

We denote by $\clP_n(R)$ ($\F_n(R)$, $\I_n(R)$) the class of right $R$-modules of projective (flat, injective) dimension at most $n$, or simply $\clP_n$ ($\F_n$, $\I_n$) when the ring is clear from the context. We let $\clP_n(\modr R)$ denote the intersection of $\modR$ and $\clP_n(R)$. The projective dimension (weak or flat dimension, injective dimension) of a right $R$-module $M$ is denoted $\pdim M_R$ ($\wdim M_R$, $\idim M_R$).

Given a ring $R$, the \emph{right big finitistic dimension}, $\Fdim R$, is the supremum of the projective dimension of right $R$-modules with finite projective dimension. The \emph{right little finitistic dimension}, f.dim $R$, is the supremum of the projective dimension of right $R$-modules in $\modr R$ with finite projective dimension. 

For an $R$-module $C$, we let $\Add(C)$ denote the class of $R$-modules which are direct summands of direct sums of copies of $C$, and $\Gen(C)$ the class of $R$-modules which are homomorphic images of direct sums of copies of $C$.%

Recall that $A$ is a \emph{pure submodule} of a right $R$-module $B$, or $A\subseteq_\ast B$, if for each finitely presented right module $F$, the functor $\Hom_R(F, -)$ is exact when applied to the short exact sequence $(1)\quad 0\to A\to B\to B/A\to 0$ or equivalently, when for every left $R$-module $M$ the functor $(- \otimes_R M)$ is exact when applied to the sequence $(1)$. 
The embedding $A \hookrightarrow B$ is called a \emph{pure embedding}, the epimorphism $B \twoheadrightarrow B/A$ a \emph{pure epimorphism} and the short exact sequence $(1)$ a \emph{pure exact sequence}. 

Short exact sequences arising from the canonical presentation of a direct limit form an important class of examples of pure exact sequences.

\begin{expl}\label{ex:lim-pure-split}
Let $(M_i, f_{ji}\mid i,j \in I)$ be a direct system of modules and consider the direct limit $\varinjlim_I M_i $. The canonical presentation 
\[
0 \to \Ker \pi \to \bigoplus_{i \in I} M_i \overset{\pi}\to \varinjlim_I M_i \to 0
\]
of $\varinjlim_I M_i$ is an example of a pure exact sequence (see e.g. \cite[Corollary 2.9]{GT12}).
\end{expl}
A module $X$ is called \emph{$\Sigma$-pure-split} if every pure embedding $A\subseteq_\ast B$ with $B \in \Add (X)$ splits.

\subsection{Homological formulae}
The following facts will be useful.
Let $F_R$ be a right $R$-module $_RG_S$ be an $R$-$S$-bimodule such that $\Tor_1^R(F, G)=0$. Then, for every right $S$-module $M_S$ there is the following injective map of abelian groups.
\begin{equation}\label{eq:water}
\Ext^1_R(F, \Hom_S(G, M))\hookrightarrow\Ext^1_S(F\otimes_RG, M))
\end{equation}

Let $f\colon R \to S$ be a ring homomorphism. Suppose $\Tor^R_i(M,S)=0$ for $M \in \ModR$ for all $1 \leq i \leq n$ and $N_S$ is a right $S$-module (and also a right $R$-module via the restriction of scalars functor). Then the following holds for all $i$ such that $1 \leq i \leq n$ (see for example \cite[Lemma 4.2]{PSlav}.
\begin{equation}\label{eq:Ext-form}
\Ext^i_R(M_R, N_R) \cong \Ext^i_S(M_R \otimes_R S, N_S)
\end{equation}
Moreover, if $M$ is as above and $N$ is a left $S$-module, then the following holds. 
\begin{equation}\label{eq:Tor-form}
\Tor^R_i(M_R, _RN) \cong \Tor^S_i(M_R \otimes _R S, _S N)
\end{equation}

\subsection{Covers, precovers and cotorsion pairs}
For this section, $\C$ will be a class of right $R$-modules closed under isomorphisms and direct summands. We recall the definitions of precovers and covers as well as some properties of covers and covering classes. 

Many of the results in this section are taken from Xu's book \cite{Xu}, which generalises work based on Enochs paper \cite{Eno} where he works mainly in the setting where $\C$ is the class of injective modules or flat modules. For this reason, many results are attributed to Enochs-Xu rather than just Enochs.

\begin{defn}A \emph{$\C$-precover} of $M$ is a homomorphism $\phi\colon C \to M$ where $C \in \C$ with the property that for every homomorphism $f\colon C' \to M$ where $C' \in \C$, there exists $f'\colon C' \to C$ such that $\phi f' = f$. 
\begin{equation*}
\xymatrix{
 C' \ar[d]_{\exists f'} \ar[dr]^f&\\
C \ar[r]_\phi & M }
\end{equation*}

A \emph{$\C$-cover} of $M$ is a $\C$-precover with the additional property that for every homomorphism $f\colon C \to C$ such that $\phi f = \phi$, $f$ is an isomorphism.
\begin{equation*}
\xymatrix{
 C \ar[d]_{f}^\cong \ar[dr]^\phi&\\
C \ar[r]_\phi & M }
\end{equation*}
\end{defn}
A $\C$-precover $\phi\colon C \to M$ of $M$ is called a \emph{special $\C$-precover} if $\phi$ is an epimorphism and $\Ker \phi \in \C^\perp$. 

If every $R$-module has a $\C$-cover ($\C$-precover, special $\C$-precover), the class $\C$ is called \emph{covering} (respectively, \emph{precovering}, \emph{special precovering}). If a cover does exist for a module $M$, we can describe the relationship between a $\C$-cover and a $\C$-precover of $M$. 

\begin{thm}\cite[Theorem 1.2.7]{Xu} \label{T:Xu-cov}
Suppose $\C$ is a class of modules and $M$ admits a $\C$-cover and $\phi\colon C \to M$ is a $\C$-precover. Then $C = C' \oplus K$ for submodules $C', K$ of $C$ such that the restriction $\phi_{\restriction_{C'}}$ gives rise to a $\C$-cover of $M$ and $K \subseteq \Ker \phi$.
\end{thm}

\begin{cor}\cite[Corollary 1.2.8]{Xu}\label{C:Xu-cov-cor}
Suppose $M$ admits a $\C$-cover. Then a $\C$-precover $\phi\colon C \to M$ is a $\C$-cover if and only if there is no non-zero direct summand $K$ of $C$ contained in $\Ker \phi$.
\end{cor}

The following two theorems will be useful when working with covers.

\begin{thm}\cite[Theorem 1.4.7, Theorem 1.4.1]{Xu}\label{T:Xu-sums-cov} 
\begin{enumerate}
\item[(i)] Suppose for each integer $n \geq 1$, $\phi_n\colon C_n \to M_n$ is a $\C$-cover. Then if $\bigoplus_n \phi_n \colon \bigoplus_n C_n \to \bigoplus_n M_n$ is a $\C$-precover, then it is also a $\C$-cover. 
 \item[(ii)] Assume that $\bigoplus \mu_n\colon \bigoplus_nC_n\to \bigoplus_nM_n$ is a $\C$-cover of $\bigoplus_nM_n$ and let $f_n\colon C_n\to C_{n+1}$ be a family of homomorphisms such that $\Img f_n\subseteq \Ker \phi_{n+1}$.
 Then, for each $x\in C_1$ there is an integer $m$ such that $f_m f_{m-1} \dots f_1(x)=0$.
 \end{enumerate}
\end{thm}

A pair of classes of modules $(\A, \B)$ is a \emph{cotorsion pair} provided that $\mbox{$\A$} = {}^{\perp_1}
\mbox{$\B$}
$ and $\mbox{$\B$} = \mbox{$\A$} ^{\perp_1}$. A cotorsion pair is called \emph{complete} if $\B$ is special preenveloping or equivalently $\A$ is special precovering. A famous result due to Eklof-Trlifaj states that if $\clS$ is a set, the cotorsion pair  $({}^{\perp_1}(\clS^{\perp_1}), \clS^{\perp_1})$ generated by $\clS$ is complete (see \cite[Theorem 6.11]{GT12}).

A cotorsion pair $(\A, \B)$ is called \emph{hereditary} if $\Ext^i_R(A,B)=0$ for every $A \in \A$, $B \in \B$ and $i > 0$ (see \cite[Lemma 5.24]{GT12}). Thus if a cotorsion pair $(\A, \B)$ is hereditary, then $\A={^\perp \B}$ and $\B=\A ^\perp $, thus there is no need to differentiate between $\perp_1$ and $\perp$.

A cotorsion pair $(\A, \B)$ is of \emph{finite type} if there is a set $\clS$ of modules in $\modr R$ such that $\clS^\perp =\B$ (recall $\modr R$ denotes the class of modules admitting a projective resolution consisting of finitely generated projective modules). In other words, $(\A, \B)$ is of finite 
type if and only if $\B=(\A\cap \modr R)^\perp$. 

\subsection{Perfect rings and projective covers}Before giving a characterisation of perfect commutative rings, we must recall some definitions.

One can generalise the notion of a nilpotent ideal to a $T$-nilpotent ideal where the $T$ stands for ``transfinite.'' An ideal $I$ of $R$ is said to be \emph{right $T$-nilpotent} if for every sequence of elements $a_1, a_2, ..., a_i, ...$ in $I$, there exists an $n >0$ such that $a_n a_{n-1} \cdots a_1 =0$. For \emph{left $T$-nilpotence}, one must have $a_1 a_{2} \cdots a_n =0$.

The property of $T$-nilpotence of an ideal has interesting consequences. In particular, an ideal $I$ is right T-nilpotent if and only if for every non-zero right $R$-module $M$, $MI$ is superfluous in $M$, $MI \ll M$ (see \cite[Lemma 28.3]{AF}).


 Let $J(R)$ denote the Jacobson radical of $R$. First recall that a ring $R$ is \emph{semilocal} if $R/J(R)$ is semisimple. If $R$ is commutative, then $R$ is semilocal if and only if it has only finitely many maximal ideals. A ring $R$ is \emph{semiartinian} if every non-zero factor of $R$ contains a simple $R$-submodule. 

The following proposition is a composite of well-known characterisations of commutative perfect rings (see for example \cite{Bass}, \cite{Lam}).
\begin{prop} \label{P:perfect}
Suppose $R$ is a commutative ring. The following statements are equivalent for $R$.
\begin{itemize}
\item[(i)] $R$ is perfect (that is, every $R$-module has a projective cover).
\item[(ii)] $\Fdim R =0$. 
\item[(iii)] $R$ is a finite product of local rings, each one with a $T$-nilpotent maximal ideal.
\item[(iv)] $R$ is semilocal and semiartinian, i.e., $R$ has only finitely many maximal ideals and every non-zero factor of $R$ contains a simple $R$-module.
\end{itemize}
Additionally, if $R$ is perfect then every element of $R$ is either a unit or a zero-divisor.
\end{prop}

 It was noticed by Bass in \cite{Bass} that it is sufficient to look at the following nice class of modules to decide if the ring is perfect. 
 
If $R$ is a ring and $\{ a_1, a_2, \dots, a_n, \dots\}$ is a sequence of elements of $R$, a \emph{Bass right $R$-module} is a flat module of the following form.
\[F=\varinjlim(R\overset{a_1}\to R\overset{a_2}\to R\overset{a_3}\to\cdots).\]
That is, $F$ is the direct limit of the direct system obtained by considering the left multiplications by the elements $a_i$ on $R$.
A direct limit presentation of $F$ is given by the following short exact sequence.
\[0\to \bigoplus_{n\in \bbN}R\overset{\sigma}\to \bigoplus_{n\in \bbN}R\to F\to 0\]
By the above projective presentation, it is clear that all Bass $R$-modules have projective dimension at most one. Thus the class of Bass $R$-modules is contained in $\F_0(R) \cap \clP_1(R)$.
The following result is well known and it is implicitly proved in Bass' paper \cite{Bass}. 

\begin{lem}\label{L:Bass-mod-perf}
Let $R$ be a ring.
\begin{enumerate}
\item[(i)] If all flat right $R$-modules have projective covers, then all the flat right $R$-modules are projective, so the ring is right perfect. 
\item[(ii)] If all Bass right $R$-modules have projective covers then the ring $R$ is right perfect. 
\end{enumerate}
\end{lem}

Recall that the \emph{socle} of a module $M$, denoted $\soc(M)$, is the sum of its simple submodules. A module $M$ is \emph{semiartinian} if every non zero quotient of $M$ has a non zero socle. Semiartinian modules are also called \emph{Loewy modules} since they admit a \emph{Loewy series}, that is a continuous filtration by semisimple (or even simple) modules constructed by transfinite induction. Thus if $R$ is a perfect commutative ring, then every module is a Loewy module by Proposition~\ref{P:perfect} (iv).

It will be useful to observe that the notion of superfluous subobject and of projective covers can be generalised from the category of $R$-modules to an arbitrary abelian category, as pointed out in Section 3 of \cite{Pos5}. 

Let $\A$ be an abelian category with enough projective objects. A subobject $B$ of an object $A$ in $\A$ is called \emph{superfluous} if for every subobject $H$ of $A$ such that $B+H=A$ one has $H=A$. Then, an epimorphism $h\colon P \to C$ with $P$ a projective object in $\A$ is a \emph{projective cover} of the object $C$ if $\Ker h$ is superfluous in $P$.

\subsection{$1$-tilting cotorsion pairs}
We now introduce $1$-tilting classes and modules, as well as some properties that we will use. 

A right $R$-module $T$ is \emph{$1$-tilting} if the following conditions hold (as defined in \cite{CT}).
\begin{enumerate}
\item[(T1)] $\pdim_R T \leq1$.
\item[(T2)] $\Ext_R^i (T, T^{(\kappa)}) =0$ for every cardinal $\kappa$ and every $i >0$.
\item[(T3)] There exists an exact sequence of the following form where $T_0, T_1$ are modules in $\Add(T)$.
\[
0 \to R \to T_0 \to T_1 \to 0
\]
\end{enumerate}
Equivalently, a module $T$ is $1$-tilting if and only if $T^{\perp_1} = \Gen(T)$ (\cite[Proposition 1.3]{CT}). The cotorsion pair generated by $T$, $({}^\perp(T^{\perp}), T^\perp)$, is called a \emph{$1$-tilting cotorsion pair} and the torsion class $T^\perp$ is called the \emph{$1$-tilting class}. Two $1$-tilting modules $T$ and $T'$ are \emph{equivalent} if they define the same $1$-tilting class, that is $T^\perp = T'^\perp$
(equivalently, if $\Add (T)=\Add(T')$). If $T$ is a $1$-tilting module which generates a $1$-tilting class $\T$, then we say that $T$ is a \emph{$1$-tilting module associated to $\T$}. 

The 
class $\T \cap {}^\perp \T$ coincides with $\Add(T)$ (see \cite[Lemma 13.10]{GT12}). As the $1$-tilting cotorsion pair is generated by a set the tilting cotorsion pair is complete by \cite[Theorem 6.11]{GT12}. Also, it is hereditary as the right-hand class $\T = \Gen (T) = T^\perp$ is clearly closed under epimorphic images, so is a coresolving class.
 Moreover, by \cite{BH08}, the $1$-tilting cotorsion pair $({}^\perp \T, \T)$ is of finite type.

The following proposition gives a necessary and sufficient condition for the left-hand side of a $1$-tilting cotorsion pair to be closed under direct limits. 

\begin{prop}\label{P:indlim-sigmapuresplit}\cite[Proposition 13.55]{GT12}
Let $T$ be a tilting module with $(\A, \T)$ the associated tilting cotorsion pair. Then $\A$ is closed under direct limits if and only if $T$ is $\Sigma$-pure-split.
\end{prop}

\section{Gabriel topologies}\label{S:gab}

In this section we recall the notions of torsion pairs and Gabriel topologies as well as proving some results that will be useful to us later on. We will conclude by discussing some advancements that relate Gabriel topologies to $1$-tilting classes over commutative rings. The reference for this section, in particular for torsion pairs and Gabriel topologies, is Stenstr\"om's book \cite[Chapters VI and IX]{Ste75}.

We will start by giving definitions in the case of a general ring with unit (not necessarily commutative). Everything will be done with reference to right $R$-modules (and right Gabriel topologies), but everything can be done analogously for left $R$-modules.

A \emph{torsion pair} $(\E, \F)$ in $\ModR$ is a pair of classes of modules in $\ModR$ which are mutually orthogonal with respect to the $\Hom$-functor and maximal with respect to this property. That is, $\E= \{M \mid \Hom_R(M, F)=0 \text { for every } F \in \F \}$ and $\F= \{M \mid \Hom_R(X, M)=0 \text{ for every } X \in \E \}$.
The class $\E$ is called a \emph{torsion class} and $\F$ a \emph{torsion-free class}.

Torsion and torsion-free classes are characterised by closure properties: A
 class $\C$ of modules is a {torsion class} if and only if it is closed under extensions, direct sums, and epimorphic images, and $\C$ is a {torsion-free class} if and only if it is closed under extensions, direct products and submodules \cite[Propositions VI.2.1 and VI.2.2]{Ste75}.
  A torsion pair $(\E, \F)$ is called \emph{hereditary} if the torsion class $\E$ is also closed under submodules, which is equivalent to $\F$ being closed under injective envelopes.

A ring $R$ is a \emph{topological ring} if it has a topology such that the ring operations are continuous. A topological ring $R$ is \emph{right linearly topological} if it has a topology with a basis of neighbourhoods of zero consisting of right ideals of $R$. A ring $R$ with a right Gabriel topology is an example of a right linearly topological ring (see \cite[Section VI.4]{Ste75}).

 A \emph{right Gabriel topology $\G$} on a ring $R$ is a filter of open right ideals in a right linear topology on $R$ satisfying an extra condition.
 This condition is such to guarantee that there is a bijective correspondence between right Gabriel topologies $\G$ on $R$ and hereditary torsion classes in $\ModR$ (see \cite[Theorem VI.5.1]{Ste75}).
 
The bijection is given by the following assignments.
\[
\xymatrix@R=8pt{
\bigg\{\txt{right Gabriel topologies \\ on $R$}\bigg\} \ar@<0.7ex>[r]^\Phi & \bigg\{\txt {hereditary torsion \\classes in $\ModR$} \bigg\} \ar@<0.7ex>[l]^\Psi \\
\Phi\colon \G \ar@{|->}[r] & \E_\G= \{ M\mid \Ann x \in \G, \forall x \in M \}\\
\{ J \leq R\mid R/J \in \E \} & \E \colon \Psi \ar@{|->}[l] }
\]
\\
The torsion pair corresponding to a Gabriel topology $\G$ will be denoted by $(\E_\G, \F_\G)$. It is generated by the cyclic modules $R/J$ where $J \in \G$, so $\F_\G$ consists of the modules $N$ such that $\Hom_R(R/J, N)=0$ for every $J\in \G$. The classes $\E_\G$ and $\F_\G$ are referred to as the \emph{$\G$-torsion} and \emph{$\G$-torsion-free} classes, respectively.

For a right $R$-module $M$ let $t_\G$ denote the associated (left exact) radical, thus $t_\G(M)$ is the maximal $\G$-torsion submodule of $M$, or sometimes $t(M)$ when the Gabriel topology is clear from context.
\subsection{Modules of quotients}\label{SS:mod-quot-gab}
A right Gabriel topology allows us to generalise localisations of commutative rings with respect to a multiplicative subset to the non-commutative setting.

The \emph{module of quotients} of the Gabriel topology $\G$ of a right $R$-module $M$ is the module $M_\G$ defined as follows. 
\[
M_\G := \varinjlim_{\substack{J \in \G}} \Hom_R(J, M/t_\G(M))
\]
Furthermore, there is the following canonical homomorphism. 
\[\psi_M\colon M\cong \Hom_R(R, M) \to M_\G\]
For each $R$-module $M$, the homomorphism $\psi_M$ is part of the following exact sequence, where both the kernel and cokernel of the map $\psi_M$ are $\G$-torsion $R$-modules. 
\begin{equation}\label{eq:psi-seq}
0 \to t_\G(M) \to M \overset{\psi_M}\to M_\G \to M_\G/\psi_M(M) \to 0
\end{equation}
By substituting $M=R$, the assignment gives a ring homomorphism $\psi_R\colon R \to R_\G$ and furthermore, for each $R$-module $M$ the module $M_\G$ is both an $R$-module and an $R_\G$-module.

Let $q\colon \ModR \to \ModR_\G$ denote the functor that maps each $M$ to its module of quotients $M_\G$. Then the $\psi$ can be considered a natural transformation of endofunctors of $\ModR$, that is the following diagram commutes.
 \begin{equation}
  \xymatrix{
 M \ar[r]^f \ar[d]^{\psi_M}& N \ar[d]^{\psi_N} \\
 M_\G  \ar[r]^{f_\G}& N_\G}
\end{equation}

A right $R$-module $M$ is \emph{$\G$-closed} if the natural homomorphism
\[M \cong \Hom_R(R, M) \to \Hom_R (J, M)\] is an isomorphism for each $J \in \G$. This amounts to saying that $\Hom_R(R/J,M) =0$ for every $J \in \G$ (i.e. $M$ is \emph{$\G$-torsion-free}) and $\Ext^1_R(R/J,M) =0$ for every $J \in \G$ (i.e. $M$ is \emph{$\G$-injective}). Moreover, if $M$ is $\G$-closed then $M$ is isomorphic to its module of quotients $M_\G$ via $\psi_M$. Conversely, every $R$-module of the form $M_\G$ is $\G$-closed. The $\G$-closed modules form a full subcategory of both $\ModR$ and $\ModR_\G$. Additionally, every $R$-linear morphism of $\G$-closed modules is also $R_\G$-linear, 

A left $R$-module $N$ is called \emph{$\G$-divisible} if $JN = N$ for every $J\in \G$. Equivalently, $N$ is $\G$-divisible if and only if $R/J \otimes_R N =0$ for each $J \in \G$. We denote the class of $\G$-divisible modules by $\D_\G$. It is straightforward to see that $\D_\G$ is a torsion class in $\RMod$.

A right Gabriel topology is \emph{faithful} if $\Hom_R(R/J, R) =0$ for every $J \in \G$, or equivalently if $R$ is $\G$-torsion-free, that is the natural map $\psi_R\colon R \to R_\G$ is injective. 

A right Gabriel topology is \emph{finitely generated} if it has a basis consisting of finitely generated right ideals. Equivalently, $\G$ is finitely generated if the $\G$-torsion radical preserves direct limits (that is there is a natural isomorphism $t_\G(\varinjlim_i M_i) \cong \varinjlim_i (t_\G(M_i))$) if and only if the $\G$-torsion-free modules are closed under direct limits (that is, the associated torsion pair is of finite type). The first of these two equivalences was shown in \cite[Proposition XIII.1.2]{Ste75}, while the second was noted by Hrbek in the discussion before \cite[Lemma 2.4]{H}.

\subsection{Perfect localisations}\label{SS:perf-loc-prelim}
There is a special class of right Gabriel topologies, called perfect right Gabriel topologies, which behave particularly well and are related to ring epimorphisms. The standard examples of these Gabriel topologies over $R$ are localisations of $R$ with respect to a multiplicative subset. 

We note that the adjective ``perfect'' for a Gabriel topology can be slightly confusing as it is not related in any way to perfect rings. However, we will continue to use this nomenclature as it is already commonly used in the literature. 

We must begin with some definitions.

A \emph{ring epimorphism} is a ring homomorphism $R \overset{u}\to U$ such that $u$ is an epimorphism in the category of unital rings. This is equivalent to the natural map $ U \otimes_R U \to U$ induced by the multiplication in $U$ being an isomorphism as $R$-$R$-bimodules (see \cite[Section XI.1]{Ste75}. We note that if $R$ is commutative and $u\colon R \to U$ a ring epimorphism, then also $U$ is commutative by \cite[Corollary 1.2]{Sil67}. 

Two ring epimorphisms $R \overset{u}\to U$ and $R \overset{u'}\to U'$ are equivalent if there is a ring isomorphism $\sigma\colon U\to U'$ such that $\sigma u=u'$.

A ring epimorphism is called \emph{(left) flat} if $u$ makes $U$ into a flat left $R$-module. 
We will denote the cokernel of $u$ by $K$ and sometimes by $U/R$ or $U/u(R)$.

A left flat ring epimorphism $R \overset{u}\to U$ is called a \emph{perfect right localisation} of $R$. In this case, by \cite[Theorem XI.2.1]{Ste75} the family of right ideals
\[
\G = \{ J \leq R \mid J U = U \} 
\]
forms a right Gabriel topology. Moreover, there is a ring isomorphism $\sigma:U \to R_\G$ such that $\sigma u: R \to R_\G$ is the canonical isomorphism $\psi_R: R \to R_\G$, or, in other words, $u$ and $\psi_R$ are equivalent ring epimorphisms. Note also that a right ideal $J$ of $R$ is in $\G$ if and only if $R/J \otimes_R U =0$.

We will make use of the characterisations of perfect right localisations from Proposition XI.3.4 of Stenstr\"om's book \cite{Ste75}. 

\subsection{Gabriel topologies and $1$-tilting classes}
As mentioned before, our work relies on a characterisation of $1$-tilting cotorsion pairs over commutative rings. Specifically, in \cite{H}, Hrbek showed that over commutative rings the faithful finitely generated Gabriel topologies are in bijective correspondence with $1$-tilting classes, and that the latter are exactly the classes of $\G$-divisible modules of the associated faithful finitely generated Gabriel topology $\G$.

The following theorem is an indispensable starting point for this paper. 
\begin{thm} \cite[Theorem 3.16]{H} \label{T:Hrb-tilting}
Let $R$ be a commutative ring. There are bijections between the following collections.
\begin{enumerate}
\item[(i)] $1$-tilting classes $\T$ in $\ModR$. 
\item[(ii)] Faithful finitely generated Gabriel topologies $\G$ on $R$.
\item[(iii)] Faithful hereditary torsion pairs $(\E,\F)$ of finite type in $\ModR$.
\end{enumerate}
Moreover, the tilting class $\T$ is the class of $\G$-divisible modules with respect to the associated Gabriel topology $\G$ of $\T$.
\end{thm}

When we refer to the \emph{Gabriel topology associated to a $1$-tilting class $\T$} we will always mean the Gabriel topology in the sense of the above theorem. We will denote by $\D_\G$ the $1$-tilting class associated to $\G$ and by $\A$ the left $\Ext$-orthogonal class to $\D_\G$ so $(\A, \D_\G)$ will denote the $1$-tilting cotorsion pair associated to $\G$.

Moreover, in the case of a Gabriel topology that arises from a perfect localisation such that $\pdim R_\G \leq 1$, it is possible to describe the $1$-tilting class more explicitly as seen in the following proposition. 
This observation is crucial as the $1$-tilting module $R_\G \oplus R_\G/R$ is much more convenient to work with than the $1$-tilting class $\D_\G$.

\begin{prop}\cite[Proposition 5.4]{H}\label{P:Hrb-5.4}
Let R be a commutative ring, $T$ a $1$-tilting module, and $\G$ the Gabriel topology associated to the $1$-tilting class $\D_\G = T^\perp$ in the sense of Theorem~\ref{T:Hrb-tilting}. Then the following are equivalent.
\begin{enumerate}
\item[(i)] $\G$ is a perfect Gabriel topology and $\pdim R_\G \leq 1$.
\item[(ii)] $R_\G \oplus R_\G/R$ is a $1$-tilting module for $\D_\G$.
\item[(iii)] $\Gen (R_\G) = \D_\G$
\end{enumerate}
If the above equivalent conditions hold, $T$ or the $1$-tilting class $\D_\G$ is said to arise from a perfect localisation.
\end{prop}
We note that there is yet more confusion with our terminology. That is the $1$-tilting class arises from a perfect localisation if and only if the Gabriel topology arises from a perfect localisation {\em and} $\pdim R_\G \leq 1$. Therefore we often include the statement $\pdim R_\G \leq 1$ for clarity.

\subsection{More properties of Gabriel topologies}\label{SS:gab-prop}
We refer to \cite{BLG} for more properties of right Gabriel topologies. Many of them hold in the non-commutative case.
 In particular we will use the following results.
 \begin{lem}\cite[Lemma 4.1, Lemma 4.2]{BLG}\label{L:G-top-facts}
Suppose $\G$ is a right Gabriel topology. Then the following statements hold. 
\begin{enumerate}
\item[(i)] Suppose $R$ is commutative. If an $R$-module $D$ is both $\G$-divisible and $\G$-torsion-free, then $D$ is an $R_\G$-module and $D \cong D\otimes_R R_\G$ via the natural map $\id_D \otimes_R \psi_R\colon D \otimes_R R \to D \otimes_R R_\G$.
\item[(ii)] If $\pdim_R M\leq 1$, then $\Tor^R_1(M, R_\G)=0$.
\end{enumerate}
\end{lem}


We will often refer to the following exact sequence where $\psi_R$ is the ring of quotients homomorphism discussed in Subsection~\ref{SS:mod-quot-gab}. 
 
\begin{equation}\label{eq:psiR}
0 \to t_\G(R) \to R \overset{\psi_R} \to R_\G \to R_\G/\psi_R(R) \to 0
\end{equation}
We will denote $t_\G(M)$ simply by $t(M)$ and when clear from the context, $\psi$ instead of $\psi_R$.
We add the following result.
\begin{lem}\label{L:ext-tors-closed}
Consider a right Gabriel topology $\G$. Let $M$ be a $\G$-torsion module and $N$ a $\G$-closed module in $\ModR$. Then $\Ext^1_R(M,N)=0$.
\end{lem}
\begin{proof}
Let $\G$ be a Gabriel topology of right ideals and $\E$ its associated hereditary torsion class in $\ModR$ which is generated by the cyclic modules $R/J$ where $J \in \G$. Therefore, for $M$ a $\G$-torsion module, there exists a presentation of $M$ as follows. 
\begin{equation}\label{E:bread}
0 \to H \to \bigoplus_{J_\alpha\in \G} R/J_\alpha \to M \to 0
\end{equation}
The module $H$ is $\G$-torsion since $\E$ is a hereditary torsion class. Take a $\G$-closed module $N$ and apply the functor $\Hom_R(-, N)$ to (\ref{E:bread}).
\begin{equation}\label{E:tea}
0= \Hom_R(H, N) \to \Ext^1_R(M, N) \to \Ext^1_R(\bigoplus R/J_\alpha, N)=0
\end{equation}
The first abelian group of the sequence (\ref{E:tea}) vanishes since $H$ is $\G$-torsion and the last abelian group vanishes since direct sums commute with $\Ext^i_R(-, N)$ and $\Ext^1_R(R/J_\alpha,N)=0$ for every $J_\alpha \in \G$. Therefore $\Ext^1_R(M, N)=0$ as desired.
\end{proof}

The following lemma is taken from \cite[Exercise IX.1.4]{Ste75}, although we state the result in a slightly more convenient way for us and include (iii) and (iv). We let $E(M)$ denote the injective envelope of $M$.
\begin{lem} \label{L:stenex}
Let $\G$ be a right Gabriel topology on $R$. Then the following are equivalent.
\begin{enumerate}
\item[(i)] The functor $q\colon \ModR \to \ModR_\G$ which maps each module $M$ to the $\G$-closed module $M_\G$ is exact.
\item[(ii)] The module $E(M)/M$ is $\G$-closed for every $\G$-closed module $M$.
\item[(iii)] For every $\G$-closed module $M$ and each $J \in \G$, $\Ext^2_R(R/J,M)=0$. 
\item[(iv)] For every $\G$-closed module $M$ and basis element $J \in \G$, $\Ext^2_R(R/J,M)=0$. 
\end{enumerate}
\end{lem}
\begin{proof}[Sketch of proof]
We will prove (i)$\Rightarrow$(ii)$\Leftrightarrow$(iii)$\Rightarrow$(i) and (iii)$\Leftrightarrow$(iv).
The equivalence of (ii) and (iii) follows from the isomorphism $
\Ext^1_R(R/J,E(M)/M)\cong \Ext^2_R(R/J,M)$.

For the implication (i) $\Rightarrow$ (ii), we begin by assuming that $q$ is exact. Fix a $J\in \G$ and take a $\G$-closed $R$-module $M$. Then since $M$ is essential in its injective envelope $E(M)$, $E(M)$ must be $\G$-torsion-free and thus is $\G$-closed. Thus we have the following commuting diagram, where the exactness of the bottom row follows by our assumption that $q$ is exact.
\[
\xymatrix{
0 \ar[r] & M \ar[r] \ar[d]^{\psi_M}_\cong & E(M) \ar[r] \ar[d]^{\psi_{E(M)}}_\cong& E(M)/M \ar[r]\ar[d]^{\psi_{E(M)/M}} & 0\\
0 \ar[r] & M_\G \ar[r] & E(M)_\G \ar[r]^{} & (E(M)/M)_\G \ar[r] & 0}
\]
 It follows by the snake lemma that $E(M)/M$ is isomorphic to its module of quotients so is $\G$-closed.

Now we show that (iii)$\Rightarrow$(i). Assume that for every $\G$-closed module $M$ and every $J \in \G$, $\Ext^2_R(R/J,M)=0$, therefore it follows that the $\G$-closed modules are closed under cokernels of monomorphisms. Now consider $q$ applied to the exact sequence $0 \to L \overset{f}\to M \overset{g}\to N \to 0$. Recall that $q$ is left exact, so it remains only to show that the induced map $g_\G$ is a surjection. We have the following commuting diagram where the bottom row is also in $\ModR_\G$.
 \begin{equation}\label{E:q-functor-1} 
  \xymatrix{
0 \ar[r] & L \ar[r]^f \ar[d]^{\psi_L} & M \ar[r]^g \ar[d]^{\psi_M}& N \ar[r]\ar[d]^{\psi_N} & 0\\
0 \ar[r] & L_\G \ar[r]^{f_\G} & M_\G  \ar[r]^{g_\G}& N_\G}
\end{equation}
 
It is sufficient to show that $g_\G$ is a surjection. By assumption $\Coker f_\G$ is $\G$-closed, and one uses this to show that $N_\G$ and $\Coker f_\G$ coincide. One uses \cite[Proposition IX.1.11]{Ste75}, that is for any $\G$-closed module $X$, there is an isomorphism $\psi_N^\ast\colon \Hom_R(N_\G, X) \overset{\cong} \to \Hom_R(N,X)$, to show that $\psi_N$ must factor through the unique induced map $h \colon N \to \Coker f_\G$, and by another application of the proposition the induced map $k \colon \Coker f_\G \to N_\G$ is an isomorphism.

That (iii)$\Rightarrow$(iv) is trivial. For the converse, for every ideal $J \in \G$ there exists a basis element $J_0 \in \G$ such that $J_0 \subseteq J$. Thus for $M$ $\G$-closed, one applies $\Hom_R(-, M)$ to $0 \to J/J_0 \to R/J_0 \to R/J \to 0$. As $J/J_0$ is $\G$-torsion, the conclusion follows by applying Lemma~\ref{L:ext-tors-closed}.

\end{proof}
The following lemma will be useful when working with a faithful Gabriel topology over a commutative ring that arises from a perfect localisation.
 \begin{lem}\cite[Lemma 4.5]{BLG} \label{L:finmanyann}
Let $R$ be a commutative ring, $u\colon R \to U$ a flat injective ring epimorphism, and $\G$ the associated Gabriel topology. Then the annihilators of the elements of $U/R$ form a sub-basis for the Gabriel topology $\G$. That is, for every $J\in \G$ there exist $z_1, z_2, \dots , z_n \in U$ such that 
\[
\bigcap_{\substack{
   0 \leq i \leq n}}
   \Ann_R(z_i +R) \subseteq J.\] 
 \end{lem}

\section{When $\A$ is covering $\G$ arises from a perfect localisation and $R_\G \oplus R_\G/R$ is $1$-tilting}\label{S:cov-Rg-div}
In this section we consider the following setting.
\begin{set}\label{set:A-cov}
Let $R$ be a commutative ring and let $(\A, \D_\G)$ be a $1$-tilting cotorsion pair with associated Gabriel topology $\G$ such that $\A$ is a covering class.
\end{set}

We show first that the Gabriel topology $\G$ arises from a perfect localisation and that $\pdim R_\G\leq 1$ so that $\D_\G = \Gen (R_\G)$, in other words we show that the equivalent conditions of Proposition~\ref{P:Hrb-5.4} hold.


We begin by describing covers of modules annihilated by some $J \in \G$.

\begin{lem} \label{L:tors-cov}
Suppose $R$ is commutative and let $(\A, \D_\G)$ be a $1$-tilting cotorsion pair with associated Gabriel topology $\G$. Consider an $R$-module $M$ such that $MJ=0$ for some finitely generated $J \in \G$ and let the following be an $\A$-cover of $M$.
\[
0 \to B \to A \xrightarrow{\phi} M \to 0
\]
Then both $A$ and $B$ are $\G$-torsion.
\end{lem}
\begin{proof}
We will use the T-nilpotency of direct sums of covers as in Theorem~\ref{T:Xu-sums-cov}~(ii). 
Let $J \in \G$ be finitely generated with a generating set $\{x_1, \dots, x_t\}$ and suppose $M$ has the property that $M J=0$, and let the sequence above be an $\A$-cover of $M$.
For every $n\in \bbN$, let $B_n$, $A_n$, $M_n$ be isomorphic copies of $B$, $A$, $M$, respectively, and $\phi_n$ the homomorphism $\phi \colon A_n \to M_n$.
Consider the following countable direct sum of covers of $M$ which is a cover of $\bigoplus_{n}M_{n}$ by Theorem~\ref{T:Xu-sums-cov}~(i).
\[ 0 \to \bigoplus_{n}
B_{n} \to \bigoplus_{n} A_{n} \xrightarrow{\bigoplus \phi_n} \bigoplus_{n}M_{n} \to 0 .\] 
Choose an element $x \in J$ and for each $n$ set $f_n\colon A_n\to A_{n+1}$ to be the multiplication by $x$.

Then clearly $\phi(f_n(A_n))=0$ for every $n>0$, hence we can apply Theorem~\ref{T:Xu-sums-cov}~(ii). For every $a \in A$, there exists an $m$ such that 
\[ f_m \circ \cdots \circ f_2 \circ f_1 (a) = 0 \in A_{m+1}.\]
Hence for every $a \in A$ there is an integer $m$ for which $x^m a = 0$.

Fix $a \in A$ and let $m_i$ be the minimal natural number for which $(x_i)^{m_i}a=0$ and set $m:= \sup\{m_i\mid 1 \leq i \leq t\}$. Then for a large enough integer $k$ we have that $J^k a =0$ (for example set $k=tm$), and $J^k \in \G$. Thus every element of $A$ is annihilated by an ideal contained in $\G$, therefore $A$ is $\G$-torsion. Since the associated torsion pair of the Gabriel topology is hereditary, also $B$ is $\G$-torsion.
\end{proof}

Next we show that $\G$ must arise from a perfect localisation using an exercise from Stenstr\"om, Lemma~\ref{L:stenex}.

\begin{lem}\label{L:cov-G-perf}
Suppose $R$ is commutative and let $(\A, \D_\G)$ be a $1$-tilting cotorsion pair with associated Gabriel topology $\G$. Suppose $\A$ is covering. Then $\G$ is a perfect Gabriel topology.
\end{lem}
\begin{proof}
By \cite[Proposition XI.3.4]{Ste75}, $R_\G$ arises from a perfect localisation if and only if both the functor $q$ is exact and $\G$ has a basis of finitely generated ideals. The associated Gabriel topology, $\G$ of a $1$-tilting class has a basis of finitely generated ideals by Hrbek's characterisation in Theorem~\ref{T:Hrb-tilting}, so it remains only to show that $q$ is exact.

We will show that $\Ext^2_R(R/J, M)=0$ for every $\G$-closed $R$-module $M$ and every finitely generated $J \in \G$, and then apply Lemma~\ref{L:stenex} to conclude that $q$ is exact. 

Let $M$ be any $\G$-closed $R$-module and $J \in \G$ finitely generated, and consider the following $\A$-cover of $R/J$. 
\[
0 \to B_J \to A_J \to R/J \to 0
\]
By Lemma~\ref{L:tors-cov}, $A_J$ and $B_J$ are $\G$-torsion. We apply the contravariant functor $\Hom_R(-, M)$ to the above cover, and find the following exact sequence.
\[
0=\Ext^1_R(B_J, M) \to \Ext^2_R(R/J, M) \to \Ext^2_R(A_J, M)=0
\]
The first module $\Ext^1_R(B_J, M)$ vanishes by Lemma~\ref{L:ext-tors-closed} since $B_J$ is $\G$-torsion and $M$ is $\G$-closed. The last module $\Ext^2_R(A_J, M)$ vanishes since $\pdim A_J \leq 1$. Therefore $\Ext^2_R(R/J, M)=0$ for every $M$ $\G$-closed and every finitely generated $J \in \G$, as required.
\end{proof}

The above lemma allows us to use the equivalent conditions of \cite[Proposition XI.3.4]{Ste75}. In particular, we have that $\psi_R\colon R \to R_\G$ is a flat injective ring epimorphism and that $R_\G$ is $\G$-divisible, so $R_\G \in \D_\G$. It remains to see that if $\A$ is covering in $(\A, \D_\G)$ then $R_\G \oplus R_\G/R$ is the associated $1$-tilting module, that is the equivalent conditions of Proposition~\ref{P:Hrb-5.4}. This amounts to showing that $\pdim R_\G \leq 1$.

\begin{prop}\label{P:R_G-tilting}
Suppose $R$ is commutative and let $(\A, \D_\G)$ be a $1$-tilting cotorsion pair with associated Gabriel topology $\G$. Suppose $\A$ is covering, then $\pdim R_\G \leq 1$. In particular, the module $R_\G \oplus R_\G/R$ is a $1$-tilting module associated to the cotorsion pair $(\A, \D_\G)$ and moreover $\Gen(R_\G) = \D_\G$.
\end{prop}
\begin{proof}
We know that $\G$ is perfect, so that $R_\G$ is $\G$-divisible by Lemma~\ref{L:cov-G-perf}.We prove that  $\pdim R_\G \leq 1$ by showing that $R_\G \in \A$. Let the following be an $\A$-cover of $R_\G$.
\begin{equation}\label{E:rain}
0 \to D \to A \xrightarrow{\phi} R_\G \to 0
\end{equation}
Note that $A$ is $\G$-divisible since both $R_\G$ and $D$ are $\G$-divisible. We will first show that $A$ must be $\G$-torsion-free, and therefore an $R_\G$-module. Fix a finitely generated $J \in \G$ with generators $x_1, \dots, x_n$. We will show that $A[J]=0$, that is the only element of $A$ annihilated by $J$ is $0$. Since $R_\G$ is divisible, one can write $1_R= 1_{R_\G} = \sum x_i \eta_i$ for some $x_i \in J$ and $\eta_i \in R_\G$. Let $\mathbf{x}$, $\mathbf{s}$ and $\mathbf{s}_A$ be the following homomorphisms.
\[ \xymatrix@R=.1cm{
{\mathbf{x}}\colon R_\G \ar[r] & \bigoplus_{1 \leq i \leq n} R_\G & {\mathbf{s}}\colon \bigoplus_{1 \leq i \leq n} R_\G \ar[r] & R_\G\\
\hspace{15pt}1_{R_\G} \ar@{|-_{>}}[r] &(\eta_1, ..., \eta_n) & \hspace{15pt}(\nu_1, ..., \nu_n) \ar@{|-_{>}}[r] & \sum_i x_i\nu_i}
\]
\\
\[ \xymatrix@R=.1cm{
{\mathbf{s}_A}\colon \bigoplus_{1 \leq i \leq n} A \ar[r] & A \\
\hspace{20pt}(a_1, ..., a_n) \ar@{|-_{>}}[r] & \sum_i x_ia_i}
\]
\\
By the definition of $\mathbf{s}$ and $\mathbf{s}_A$, the lower square of (\ref{eq:bells}) commutes. Clearly $\phi^n$ is a precover of $R_\G^n$ as $D^n \in \D_\G$ and $A^n \in \A$ (it is in fact a cover). Therefore, there exists a map $f$ such that the upper square of (\ref{eq:bells}) commutes.
\begin{equation}\label{eq:bells}
 \xymatrix{
A \ar[r]^\phi \ar[d]^{f} & R_\G \ar[r] \ar[d]^{\mathbf{x}} & 0\\
A^n \ar[r]^{\oplus_n \phi} \ar[d]^{\mathbf{s}_A} & {R_\G}^n \ar[r] \ar[d]^{\mathbf{s}} & 0\\
A \ar[r]^\phi & R_\G \ar[r] & 0}
\end{equation}
The map $\mathbf{sx}$ is the identity on $R_\G$, so we have that $ \phi \mathbf{s}_A f= \phi$ and by the $\A$-cover property of $\phi$, $\mathbf{s}_A f$ is an automorphism of $A$. Consider an element $a \in A[J]$ and let $f(a) = (f_1(a), \dots, f_n(a)) \in A^n$. Then $\mathbf{s}_A (f(a)) = \sum x_if_i(a)=\sum f_i(x_ia)=0$ as $x_i \in J$, and by the injectivity of $\mathbf{s}_Af$, $a=0$.

We have shown that $A, D$ are both $\G$-torsion-free and $\G$-divisible, so by Lemma~\ref{L:G-top-facts}~(i) they are $R_\G$-modules. Then the sequence (\ref{E:rain}) is a sequence in $\ModR_\G$ as $R \to R_\G$ is a ring epimorphism and $\ModR_\G \to \ModR$ is fully faithful. Thus (\ref{E:rain}) splits so $R_\G \in \A$ and $\pdim R_\G \leq 1$ as required.

The last statement then follows by Proposition~\ref{P:Hrb-5.4}.
\end{proof}
\section{Topological rings and $u$-contramodules}\label{S:top}
The material covered in this section is a combination of ideas from \cite{Pos}, \cite{BP3}, and \cite{BP4}, and covers mostly methods using contramodules and topological rings.

An abelian group is a \emph{topological group} if it has a topology such that the group operations are continuous. A topological abelian group is said to be \emph{linearly topological} if there is a basis of neighbourhoods of zero consisting of subgroups. 

For a linearly topological abelian group $A$ with basis $\mathfrak{B}$ of subgroups of $A$, there is the following canonical homomorphism of abelian groups.
\[
\lambda_A:A \to \varprojlim_{\substack{V \in \mathfrak{B}}} A/V
\]
When $\lambda_A$ is a monomorphism, or equivalently when $\bigcap_{V \in \mathfrak{B}}V =0$, $A$ is said to be \emph{separated}. When $\lambda_A$ is an epimorphism, $A$ is said to be \emph{complete}. 

For a ring $R$ and $M, N \in \ModR$, the abelian group $\Hom_R(M,N)$ can be considered a linearly topological abelian group as follows. Take a finitely generated submodule $F$ of $M$, and consider the subgroup formed by the elements of $\Hom_R(M,N)$ which annihilate $F$. The subgroups of this form form a base of neighbourhoods of zero in $\Hom_R(M,N)$. Note that this is the same as considering $\Hom_R(M,N)$ with the subspace topology of the product topology on $N^{M}$, where the topology on $N$ is the discrete topology. We will consider $\Hom_R(M,N)$ endowed with this topology which we will call the \emph{finite topology}. The topological abelian group $\Hom_R(M,N)$ is separated and complete with respect to this topology.
\\

Recall from Section~\ref{S:gab} that a topological ring $R$ is \emph{right linearly topological} if it has a topology with a basis of neighbourhoods of zero consisting of right ideals of $R$, and that a ring $R$ with a right Gabriel topology is an example of a right linearly topological ring.

Let $\clH$ be a topology of a linearly topological commutative ring $R$ with basis $\mathfrak{ B}$. The {\em $\clH$-topology} on an $R$-module $M$ is the topology where the base of neighbourhoods of $0$ are the submodules $MJ$ for $J \in \mathfrak{B}$.
For every $R$-module $M$, $\{M/MJ\mid J\in 
\mathfrak B\}$ is an inverse system. 
 The \emph{completion} of $M$ with respect to the $\clH$-topology is the module
\[
\Lambda_{\mathfrak{ B}}(M) := \varprojlim_{\substack{J \in \mathfrak B}} M/MJ.
 \]
There is a canonical map $\lambda_M\colon M \to \Lambda_\mathfrak{B}(M)$ which sends the element $x\in M$ to $(x +MJ)_{J\in \mathfrak{ B}}$. Each element in $\Lambda_{\mathfrak{ B}}(M)$ is of the form $(x_J +MJ)_{J\in \mathfrak{ B}}$ with the relation that for $J \subseteq J'$, $x_J - x_{J'} \in MJ'$. The module $M$ is called \emph{$\clH$-separated} if the homomorphism $\lambda_M$ is injective, which is equivalent to $\bigcap_{J \in \mathfrak{ B}}MJ =0$. The module $M$ is called \emph{$\clH$-complete} if the map $\lambda_M$ is surjective.


Let $R$ be a linearly topological commutative ring with a linear topology $\clH$. A $R$-module $M$ is \emph{$\clH$-discrete} if for every $x \in M$, the annihilator ideal $\Ann_R(x) = \{r \in R\mid xr =0\}$ is open in the topology of $R$. 
In the case that the topology $\clH$ on $R$ is a Gabriel topology $\G$ on $R$, then a module is $\clH$-discrete if and only if it is $\G$-torsion.

Therefore $\psi_R\colon R \to R_\G$ is a flat injective ring epimorphism of commutative rings, which as usual we will denote by $u\colon R \to U$. For every pair of $R$-modules $M$ and $N$, the $R$-modue $\Hom_R(M,N)$ can be endowed both with the finite topology and the $\G$-topology.

For $K:=U/R$ we first show that $\Hom_R(K, M)$ is Hausdorff in the $\G$-topology. 

\begin{lem}\label{L:HomK-H-sep}
Let $R$ be a commutative ring and $\G$ a faithful perfect Gabriel topology on $R$. Then every open basis element  in the finite topology on $\Hom_R(K, M)$  contains $\Hom_R(K, M)J$ for some $J \in \G$. Hence $\Hom_R(K, M)$ is $\G$-separated for every $R$-module $M$.
\end{lem}

\begin{proof}

Fix a finitely generated submodule $X$ of $K$ and let $V_X$ be the collection of homomorphisms which annihilate $X$. Then as $X$ is a finitely generated submodule and $K$ is $\G$-torsion, there exists a $J \in \G$ such that $XJ =0$. Thus $\Hom_R(K, M)J \subseteq V_X$, as required. 

The last statement follows, since $\Hom_R(K, M)$ is always separated in the finite topology, 
\end{proof}
In particular, we will be interested in the linear topological ring $\mathfrak{R}:=\End_R(K)$ with the finite topology. Later on in Proposition~\ref{P:same-topologies} we will show that the $\G$-topology and the finite topology on $\mathfrak{R}$ coincide.

\subsection{$u$-contramodules}\label{SS:u-contra}

We will begin by considering a general commutative ring epimorphism $u\colon R \to U$ before moving onto flat injective ring epimorphisms.

A module $M$ is \emph{$u$-divisible} if $M$ is an epimorphic image of $U^{(\alpha)}$ for some cardinal $\alpha$. An $R$-module $M$ has a unique $u$-divisible submodule denoted $h_u(M)$, which is the image of the map $u^{\ast}\colon\Hom_R(U,M) \to \Hom_R(R,M) \cong M$. In nice situations, that is when $U$ is flat and $\G$ is the Gabriel topology associated to $u$, the $u$-divisible modules are $\G$-divisible. Later in this section we will discuss when these classes of modules coincide. The following definition is borrowed from \cite{BP4}.
\begin{defn}
Let $u\colon R\to U$ be a ring epimorphism. A \emph{$u$-contramodule} is an $R$-module $M$ such that the following holds. 
\[ \Hom_R(U, M) = 0 = \Ext^1_R(U,M)
\]
\end{defn}
We let $\ucontra$ denote the full subcategory of $u$-contramodules in $\ModR$. By \cite[Proposition 1.1]{GL91} the category of $u$-contramodules is closed under kernels of morphisms, extensions, infinite products and projective limits in $\ModR$ .

 The following lemma is proved in \cite {Pos} for the case of the localisation $R_S$ of $R$ at a multiplicative subset $S$. The proof can be extended easily to the case of a ring epimorphism of commutative rings and will be very useful in the sequel.

 \begin{lem}\label{L:oneten} \cite[Lemma 1.10]{Pos}
 Let $b\colon A \to B$ and $c\colon A \to C$ be two $R$-module homomorphisms such that $C$ is a $u$-contramodule and $\Ker(b)$ is a $u$-divisible $R$-module and $\Coker(b)$ is a $U$-module. Then there exists a unique homomorphism $f\colon B \to C$ such that $c =fb$.
\end{lem}
%

Let now $0 \to R \overset{u} \to U $ be a flat injective ring epimorphism of commutative rings where $U = R_\G$, $K = R_\G/R$ and $\G$ is the associated Gabriel topology $\{J \leq R \mid JU = U \}$. We will often refer to the short exact sequence
\begin{equation}\label{E:u}
0 \to R \overset{u}\to U \overset{w}\to K \to 0
\end{equation}

In general we will use $N$ to denote a $\G$-torsion-free module $R$-module, while $M$ will denote an arbitrary $R$-module.

For an $R$-module $M$, by applying the contravariant functor $\Hom_R(-,M)$ to the short exact sequence (\ref{E:u}) we have the following short exact sequences.
\begin{equation}\label{eq:1contra}
0 \to \Hom_R(K, M) \to \Hom_R(U,M) \to h_u(M) \to 0 
\end{equation}
\begin{equation} \label{eq:8contra}
0 \to h_u(M) \to M \to M/h_u(M) \to 0
\end{equation}
\begin{equation} \label{eq:2contra}
0 \to M / h_u(M) \to \Ext^1_R(K,M) \to \Ext^1_R(U,M) \to 0
\end{equation}
\emph{For an $R$-module $M$, we let $\Delta_u(M)$ denote the module $\Ext^1_R(K,M)$ and $\delta_M\colon M \to \Delta_u(M)$ the natural connecting map from the exact sequences (\ref{eq:8contra}) and (\ref{eq:2contra}).}

For each $R$-module $M$, let $\nu_M$ be the unit of the adjunction $\adjK$ evaluated at $M$. 
\[\xymatrix@R0.13cm{
\nu_M\colon M \ar[r] & \Hom_R(K, M \otimes_R K)\\
\hspace{15pt} m \ar@{|->}[r] & [m^\ast\colon z+R \to m \otimes_R (z+R)]& z \in U}
\]
For every $\G$-torsion-free $R$-module $N$ we have the exact sequence
\begin{equation}\label{eq:6contra}
0 \to N \to N \otimes_R U \to N \otimes_R K \to 0
\end{equation}
and applying the covariant functor $\Hom_R(K,-)$ to (\ref{eq:6contra}) we obtain the long exact sequence
\begin{equation}\label{eq:3contra}
\Hom_R(K, N \otimes_R U) \to \Hom_R(K,N \otimes_R K) \overset{\mu_N}\to
 \Ext^1_R(K, N) \to 
 \Ext^1_R(K,N \otimes_R U)
\end{equation}
where $\mu_N$ is the connecting homomorphism.
In the next lemmas we show that for a $\G$-torsion-free module $N$, the modules $\Hom_R(K, N \otimes_R K)$ and $\Delta_u(N)$ are isomorphic via the natural connecting homomorphism $\mu_N$, and moreover $\delta_N=\mu_N\nu_N$. 
%
\begin{lem}\label{L:Delta-tf}
Let $u\colon R \to U$ be a flat injective ring epimorphism of commutative rings and let $K:=U/R$. If $N$ is a $\G$-torsion-free $R$-module, then, using the above notation, the connecting morphism $\mu_N\colon \Hom_R(K, N \otimes_R K)\to \Delta_u(N) $ is an isomorphism.

\end{lem}
\begin{proof}

The first term in equation (\ref{eq:3contra}) vanishes as $K$ is $\G$-torsion and $N \otimes_RU$ is $\G$-torsion-free. The last term in equation (\ref{eq:3contra}) vanishes since by the flatness of the ring $U$, there is an isomorphism \[\Ext^1_R(K,N \otimes_R U)\cong \Ext^1_U(K\otimes_R U,N \otimes_R U) =0.\]Thus $\Hom_R(K, N \otimes_R K)$ is isomorphic to $\Ext^1_R(K, N)= \Delta_u(N)$ via $\mu_N$.

Alternatively, one can use Lemma~\ref{L:ext-tors-closed} as $K$ is $\G$-torsion and $N \otimes_RU$ is $\G$-closed.
\end{proof}

Before continuing with the goal of proving that $\delta_N=\mu_N\nu_N$, we state a consequence of the Lemma~\ref{L:Delta-tf}. We note that in the reference provided, the statement is more general thus requires a more sophisticated proof, whereas here we choose to provide a simpler proof. 

\begin{lem}\cite[Lemma 2.5(a),(b)]{BP3}\label{L:some-ucontras}
Let $u\colon R \to U$ be a flat injective ring epimorphism and let $K:=U/R$. Then the following hold.
\begin{enumerate}
\item[(i)] $\Hom_R(K,M)$ is a $u$-contramodule for every $R$-module $M$.
\item[(ii)] $\Delta_u(N)$ is a $u$-contramodule for every $\G$-torsion-free $R$-module $N$.
\end{enumerate}
\end{lem}
\begin{proof}
(i) By the tensor-$\Hom$ adjunction, we have theisomorphism. 
\[
 \Hom_R(U, \Hom_R(K,M)) \cong \Hom_R(U \otimes_R K, M) =0
\]
To see that $\Ext_R^1(U, \Hom_R(K,M))=0$, we use the flatness of $U$. Using $\Tor^R_1(U,K)=0$ there is an inclusion (see the homological formulas in Section~\ref{S:prelim}).
\[
\Ext_R^1(U, \Hom_R(K,M))\hookrightarrow\Ext_{R}^1(U \otimes_R K, M)=0
\]
(ii) This follows by Lemma~\ref{L:Delta-tf} and (i) of this lemma.
\end{proof}

\begin{lem}\label{L:mu-commute}
Let $u\colon R \to U$ be a flat injective ring epimorphism of commutative rings and let $K:=U/R$. For $N$ a $\G$-torsion-free module, the following diagram commutes. 
\[\xymatrix{
N \ar[r]^{\delta_N} \ar[d]_{\nu_N}& \Ext^1_R(K,N) \\
\Hom_R(K, N \otimes_R K) \ar[ru]_{\cong}^{\mu_N}}
\]
\end{lem}
\begin{proof}
The morphism $\nu_N$ is the unit of the adjunction $\adjK$ evaluated at $N$. Let $\Phi$ be the adjunction isomorphism
\[
\Phi\colon \Hom_R(N,\Hom_R(K,N\otimes_RK)) \to \Hom_R(N\otimes_RK, N\otimes_RK))\]

As $\Phi(\nu_N) = \id_{N \otimes_R K}$, it is enough to show that $\Phi(\mu_N^{-1} \delta_N) = \id_{N\otimes_R K}$, that is that $(\mu_N^{-1} \delta_N)(x)(k)=x\otimes_R k$ for every $x \in N$ and $k \in K$.

Fix $x \in N$ and $k \in K$. Consider the map $f_x\colon R\to N, 1_R\mapsto x$. Then $\delta_N(f_x)$ is the map associated to the pushout of $N \overset{f_x} \gets R \overset{u}\to U$ which is shown in the top two rows of short exact sequences of Diagram~\ref{eq:pos}. As $\mu_N$ is an isomorphism, for each extension $\zeta_x$ of $K$ by $N$, one can associate a map $\mu_N^{-1}(\zeta_x)=g_x\colon K \to N\otimes_R K$ such that the bottom two rows of short exact sequences in (\ref{eq:pos}) commute and form part of a pullback diagram. 
\begin{equation}\label{eq:pos}
 \xymatrix{
&0 \ar[r] & R \ar[r] \ar[d]^{f_x} & U \ar[r] \ar[d]& 		K \ar[r] \ar@{=}[d] & 0 \\
\zeta_x\colon &0 \ar[r] & N \ar[r] \ar@{=}[d] & 	Z_x \ar[r] \ar[d] & 	K \ar[r] \ar[d]^{g_x} & 0\\
&0 \ar[r] & N \ar[r] & 			N \otimes_R U \ar[r] & 	N \otimes_R K \ar[r] & 0}
\end{equation}
As $N\otimes_RU$ is $\G$-torsion-free and $N\otimes_RK$ is $\G$-torsion the commutativity of the larger square implies that the map $U \to Z_x \to N \otimes_R U$ is exactly the map $z \mapsto x \otimes_R z$ for $z \in U$, as this map makes the larger left square commute. Thus $g_x\colon z+R \mapsto x \otimes_R(z+R)$. It is now straightforward to see that $(\mu_N^{-1} \delta_N)(x)(k)=(\mu_N^{-1})(\zeta_x)(k)=g_x(k)=x\otimes_R k$. 
\end{proof}
%
\begin{cor}
Let $u\colon R \to U$ be a flat injective ring epimorphism of commutative rings, $K:=U/R$ and let $N$be a  $\G$-torsion-free module. Then the kernel of $\nu_N\colon N \to \Hom_R(K, N \otimes_R K)$ is $u$-divisible and the cokernel is a $U$-module.
\end{cor}
\begin{proof}
This follows from Lemma~\ref{L:mu-commute} as $\mu_N\nu_N=\delta_N$ and $\mu_N$ is an isomorphism. So $\Ker \nu_N \cong \Ker \delta_N = h_u(N)$ is a $u$-divisible module and $\Coker \nu_N = \Coker \delta_N = \Ext^1_R(U,N)$ is a $U$-module as required.
\end{proof}

The following lemma will be useful in Section~\ref{S:Gperf}. It is taken from \cite{Pos} where it is proved for the case of a localisation of a ring at a multiplicative subset. We show how to adapt the proof to our situation.
\begin{lem}\cite[Lemma 1.11]{Pos}\label{L:R-J-tens-tf}
Let $u\colon R \to U$ be a flat injective ring epimorphism with associated Gabriel topology $\G$, and let $M$ be any $R$-module. Then $M /JM\cong R/J\otimes_R \Delta_u(M) $ is an isomorphism for every $J \in \G$.
\end{lem}
\begin{proof}
Consider the equations.
\begin{equation} \tag{\ref{eq:8contra}}
0 \to h_u(M) \to M \to M/h_u(M) \to 0 \\
\end{equation}
\begin{equation} \tag{\ref{eq:2contra}}
0 \to M / h_u(M) \to \Ext^1_R(K,M) \to \Ext^1_R(U,M) \to 0
\end{equation}
Applying $(R/J \otimes_R -)$ to (\ref{eq:8contra}) we have that $M/JM \cong R/J \otimes_R M/h_u(M)$ as $h_u(M)$ is $\G$-divisible. Applying $(R/J \otimes_R -)$ to (\ref{eq:2contra}) we find \[R/J \otimes_R M/h_u(M)\cong R/J \otimes_R \Ext^1_R(K,M)\] as $\Ext^1_R(U,M)$ is a $U$-module and $\Tor^R_1(R/J, \Ext^1_R(U, M))\cong \Tor^U_1(R/J\otimes_RU, \Ext^1_R(U, M))$, since $U$ is flat.\end{proof}
The following lemma will also be useful in Section~\ref{S:Gperf}. 
\begin{lem}\label{L:Tor-R/J-Gtf}
Let $u\colon R \to U$ be a flat injective ring epimorphism with associated Gabriel topology $\G$ and let $N$ be a $\G$-torsion-free $R$-module. Then $\Tor^R_1(R/J, N) \cong \Tor^R_1(R/J, \Delta_u(N))$ is an isomorphism for every $J \in \G$.
\end{lem}

\begin{proof} 
Note first that $\Tor^R_i(R/J, Z)=0 = R/J \otimes_RZ$ for any $U$-module $Z$ and $i>0$ since $U$ is flat and so $\Tor^R_i(R/J, Z) \cong \Tor^U_i(R/J \otimes_R U, Z)=0$.

Consider the combination of the equations
\begin{equation} \tag{\ref{eq:8contra}}
0 \to h_u(N) \to N \to N/h_u(N) \to 0 \\
\end{equation}
\begin{equation} \tag{\ref{eq:2contra}}
0 \to N / h_u(N) \to \Ext^1_R(K,N)=\Delta_u(N) \to \Ext^1_R(U,N) \to 0
\end{equation}
As $N$ is $\G$-torsion-free, also $h_u(N)$ is $\G$-torsion-free and $\G$-divisible, so is a $U$-module, by Lemma~\ref{L:G-top-facts}~(i). Thus applying $(R/J\otimes_R -)$ to the above sequences, we use the observation in the first lines of this proof, and find the following isomorphisms. 
\[
\Tor^R_1(R/J, N) \cong \Tor^R_1(R/J, N/h_u(N)) \cong\Tor^R_1(R/J, \Delta_u(N))
\]
\end{proof}
We summarize in the following corollary the results obtained by Lemmas~\ref{L:Delta-tf}, ~\ref{L:R-J-tens-tf} and ~\ref{L:Tor-R/J-Gtf}.
\begin{cor}\label{C:R/J-tors-free} Let $u\colon R \to U$ be a flat injective ring epimorphism with associated Gabriel topology $\G $. Let $N$ be a $\G$-torsion-free $R$-module and $J\in \G$. Then the following hold.
\begin{enumerate}
\item[(i)] $\Hom_R(K, N\otimes_R K)\cong \Delta_u(N)$.
\item[(ii)] $R/J\otimes_R  \Delta_u(N)\cong N/JN$.
\item[(iii)] $\Tor^R_1(R/J, N) \cong \Tor^R_1(R/J, \Delta_u(N))$.
\end{enumerate}
\end{cor}

As an application we consider the endomorphism ring $\mathfrak{R}$ of $K$. Recall that by \cite[Lemma 4.1]{BP3}, $\mathfrak{R}$ is a commutative ring.

\begin{prop}\label{P:same-topologies}  Let $u\colon R \to U$ be a flat injective ring epimorphism with associated Gabriel topology $\G $ and let $K:=U/R$. Then the finite topology and the $\G$-topology on $\mathfrak{R}=\Hom_R(K,K)$ coincide.
\end{prop}
\begin{proof} Let $X$ be a finite subset of $K$ and let $V_X$ be the annihilator of $X$ in $\mathfrak{R}$, a basis element of the finite topology on $\mathfrak{R}$.
 By Lemma~\ref{L:HomK-H-sep} there is a $J\in \G$ such that $\mathfrak{R}J\subseteq V_X$, so the $\G$-topology is a finer topology than the finite topology on $\mathfrak{R}$. Thus it remains to show that for every $J\in \G$, $\mathfrak{R}J$ contains $V_X$ for some finite subset $X$ of $K$.
 
 Consider the canonical morphism $\nu_R\colon R\to \mathfrak{R}$ sending an element $r\in R$ to the multiplication by $r$ on $K$.  Then, $I=\nu_R^{-1}(V_X)$ is the annihilator of $X$ in $R$. We have that $I\in \G$, since $K$ is $\G$-torsion. Clearly, $\mathfrak{R}I\subseteq V_X$.

Now it is straightforward to see that the following diagram commutes as the vertical and horizontal arrows are induced by $\nu_R$ and $\pi$ is the natural quotient map.
\[
\xymatrix@C=2cm{
R/I \ar[r]^{\nu \otimes_R R/I} \ar[d]_\gamma& \mathfrak{R}/\mathfrak{R}I \ar[dl]^\pi\\ 
\mathfrak{R}/V_X}
\]
By (i) and (ii) of Corollary~\ref{C:R/J-tors-free}, $\mathfrak{R}\cong\Delta_u(R)$ and $\nu \otimes_R R/I$ is an isomorphism. Since $\gamma$ is a monomorphism we conclude that $\pi$ is a monomorphism, and so $V_X= \mathfrak{R}I$.  

Fix a $J \in \G$. By Lemma~\ref{L:finmanyann} there exists a finitely generated $X \subseteq K$ such that the annihilator ideal of $X$ in $R$, is contained in $J$. So $V_X \subseteq \mathfrak{R}J$.
\end{proof}
If the flat injective ring epimorphism $u\colon R \to U$ is such that $\pdim U \leq 1$, then the category $\ucontra$ is also closed under cokernels and so is an abelian category. Moreover, if $\G$ is the associated Gabriel topology, then $\pdim U \leq 1$ if and only if the $u$-divisible modules and the $\G$-divisible modules coincide by Proposition~\ref{P:Hrb-5.4}.

\begin{prop}\label{P:Delta-left-adjoint}
Let $u\colon R \to U$ be a flat injective ring epimorphism such that $\pdim U \leq 1$. 
\begin{enumerate}\item[(i)] $\iota\colon \ucontra \hookrightarrow \ModR$ is an exact embedding and the functor $\Delta_u= \Ext^1_R(K, -)$ defines a left adjoint to this embedding. 
\item[(ii)] $\Delta_u(R)$ is a projective generator of $\ucontra$. The coproduct of $X$ copies of $\Delta_u(R)$ in $\ucontra$ is $\Delta_u(R^{(X)})$ and the projective objects in $\ucontra$ are direct summands of the objects of the form $\Delta (R^{(X)})$ for some set $X$. 
\end{enumerate}
\end{prop}

\begin{proof} (i) is \cite[Proposition 3.2(b)]{BP3}. 

(ii) follows by the properties of a left adjoint to an exact functor. \end{proof}

\subsection{The equivalence of categories}\label{SS:cat-equiv}

 In \cite{BP3} it is considered the case of a (not necessarily injective nor flat nor commutative) ring epimorphism $u\colon R\to U$ such that $\Tor^R_1(U,U)=0$. In \cite[Theorem 1.2]{BP3} it is shown that the adjunction $\big( (- \otimes_R K), \Hom_R(K,-) \big)$ (where $K= U/u(R)$) defines an equivalence between the class of $u$-divisible right $u$-comodules and the class of $u$-torsion-free left $u$-contramodules.

In our situation, that is when $u$ is a flat injective epimorphism with associated Gabriel topology $\G$, the class of $u$-comodules coincides with the class of $\G$-torsion modules and the class of $u$-torsion-free modules coincides with the class of $\G$-torsion-free modules. Thus, in our setting, Theorem 1.2 in \cite{BP3} becomes:

\begin{thm}\cite[Theorem 1.2]{BP3}\label{T:cat-equiv}
Let $u\colon R \to U$ be a flat injective ring epimorphism of commutative rings. Then the restrictions of the adjoint functors $(- \otimes_R K)$ and $\Hom_R(K, -)$ are mutually inverse equivalences between the additive categories of $u$-divisible $\G$-torsion modules and $\G$-torsion-free $u$-contramodules.\\
\[
\xymatrix{
\bigg\{ \txt{$\G$-torsion-free \\$u$-contramodules} \bigg\} \ar@/^2.5pc/[r]^{(- \otimes_R K)} & \bigg\{ \txt{$u$-divisible and \\$\G$-torsion modules} \bigg\} \ar@/^2.5pc/[l]^{\Hom_R(K,-)}
}
\]
\end{thm}

%

\section{When $\A$ is covering, $R$ is $\G$-almost perfect}\label{S:cov-G-perf}

In this section we continue with the situation of Setting~\ref{set:A-cov}. 

By Proposition~\ref{P:R_G-tilting} if $(\A, \D_\G)$ is a $1$-tilting pair such that $\A$ is covering then the associated tilting module arises from a flat injective ring epimorphism $u\colon R \to U$ and $U \oplus K$ is a $1$-tilting module for $(\A, \D_\G)$, thus $\D_\G = \Gen(U)$.

In Proposition~\ref{P:Rg-perfect-ring} we prove that $R_\G$ is a perfect ring and in Proposition~\ref{P:R-J-Bass} that the rings $R/J$ are perfect for every $J \in \G$. The main result of this section is Theorem~\ref{T:cov-implies-Gperf}.

We begin by introducing the following definition.
\begin{defn}
Let $R$ be a ring with a right Gabriel topology $\G$. Then $R$ is \emph{$\G$-almost perfect} if $R_\G$ is a perfect ring and the quotient rings $R/J$ are perfect for each $J \in \G$.
\end{defn}

\begin{prop}\label{P:Rg-perfect-ring}
Suppose $R$ is commutative and let $(\A, \D_\G)$ be a $1$-tilting cotorsion pair with associated Gabriel topology $\G$. Suppose $\A$ is covering. Then $R_\G$ is a perfect ring.
\end{prop}
\begin{proof}
We will show that every $R_\G$-module has a projective cover in $\ModR_\G$. Consider $M \in \ModR_\G$ with the following short exact sequence in $\ModR_\G$. 
\begin{equation}\label{E:rain3}
0 \to L \to R_\G^{(\alpha)} \overset{\phi}\to M \to 0
\end{equation}
Then this sequence is also a short exact sequence of $R$-modules with $R_\G^{(\alpha)} \in \A$ by Proposition~\ref{P:R_G-tilting} and $L \in \D_\G$ thus it is an $\A$-precover of $M$ (as an $R$-module). By the assumption that $\A$ is covering, one can extract from the exact sequence (\ref{E:rain3}) an $\A$-cover of $M$ of the form
\begin{equation}\label{E:rain2}
0 \to L' \to P \overset{\phi'}\to M \to 0
\end{equation}
where $L'$ and $P$ are direct summands of $L$ and $R_\G^{(\alpha)}$ respectively as $R$-modules. 
An $R$-module direct summand of an $R_\G$-module is a $\G$-torsion-free $\G$-divisible module, hence it is an $R_\G$-module by Lemma~\ref{L:G-top-facts}(i). Moreover, by Lemma~\ref{L:cov-G-perf}, $\psi_R\colon R\to R_\G$ is a ring epimorphism (even flat) so $L'$ and $P$ are direct summands as $R_\G$-modules and (\ref{E:rain2}) is in $\ModR_\G$. 

So we have shown that (\ref{E:rain2}) is a $\clP_0(R_\G)$-precover of $M$ in $\ModR_\G$, which is also an $\A$-cover when considered in $\ModR$. It remains to see that it is a $\clP_0(R_\G)$-cover. Note that every $R_\G$-homomorphism $f$ such that $\phi'f = \phi'$ is also an $R$-homomorphism, and therefore $\phi'$ is an automorphism as it is an $\A$-cover. 
\end{proof}

We will now show that $R/J$ is perfect for each $J \in \G$ by showing that every Bass $R/J$-module has a $\clP_0(R/J)$-cover, that is using Lemma~\ref{L:Bass-mod-perf}.

Take $a_1, a_2, \dots, a_i, \dots$ a sequence of elements of $R$ and let $N$ be the Bass $R/J$-module with presentation as in the sequence (\ref{E:RJ-Bass}), where $(e_i)_{i\in \bbN}$ and $(f_i)_{i\in \bbN}$ are bases of the domain and codomain of $\phi$ respectively. 
\begin{equation}\label{E:RJ-Bass}
\xymatrix@R=0.25cm{
0 \ar[r]& \bigoplus_{\bbN} R/J \ar[r]^{\tilde\sigma}& \bigoplus_{\bbN} R/J \ar[r]& N \ar[r]& 0 \\
&e_i \ar@{|->}[r]& f_i - a_if_{i+1}}
\end{equation}
As the elements $a_1, a_2, \dots, a_i, \dots$ are in $R$, we can also define a Bass $R$-module, which is a lift of $N$. That is, we consider the following Bass $R$-module.
\begin{equation}\label{E:RJ-lift-Bass}
\xymatrix@R=0.25cm{
0 \ar[r]& \bigoplus_{\bbN} R \ar[r]^{\sigma}& \bigoplus_{\bbN} R \ar[r]& F \ar[r]& 0}
\end{equation}

It is clear that applying $(R/J \otimes_R -)$ to (\ref{E:RJ-lift-Bass}) will give us (\ref{E:RJ-Bass}), thus $R/J \otimes_R F = N$, where $F$ is flat. 

We will make use of results in Subsection~\ref{SS:u-contra} and the category equivalence in Theorem~\ref{T:cat-equiv}.
\begin{lem} \label{L:p-cov-u-contra}
Suppose $\A$ is covering and $F$ is a Bass $R$-module. Then the $u$-contramodule $\Hom_R(K, F \otimes_R K) $ has a projective cover in the category of $u$-contramodules. 
\end{lem}
\begin{proof}
$\Hom_R(K, F \otimes_R K) $ is a $u$-contramodule by Lemma~\ref{L:some-ucontras}~(i). 
Apply the functor $(- \otimes_R K)$ to (\ref{E:RJ-lift-Bass}) to get the exact sequence:
\begin{equation*}
\xymatrix@C=1cm{
0 \ar[r]& \bigoplus_{\bbN} K \ar[r]^{\sigma \otimes_RK}& \bigoplus_{\bbN} K \ar[r]& F \otimes_R K \ar[r]& 0 }
\end{equation*}
The above is an $\A$-precover of $F\otimes_R K$ by Proposition~\ref{P:R_G-tilting}. As by assumption $\A$ is covering, one can extract an $\A$-cover of $F\otimes_R K$ from the above sequence in the form
\begin{equation*}
\xymatrix@C=1.4cm{
0 \ar[r]& D_1 \ar[r]^{(\sigma \otimes_RK)\restriction_{D_1}}& D_0 \ar[r]^\pi& F \otimes_R K \ar[r]& 0 }
\end{equation*}
where $D_0$ and $D_1$ are direct summands of $\bigoplus_{\bbN} K$. 
Now we apply $\Hom_R(K,-)$ to the above sequence, and we claim that it is a projective cover in the category of $u$-contramodules.
\begin{equation}\label{E:three}
0 \to \Hom_R(K,D_1) \to \Hom_R(K,D_0) \overset{\rho}\to \Hom_R(K,F \otimes_R K) \to 0
\end{equation}
Firstly, $\Hom_R(K,D_1)$ and $\Hom_R(K,D_0)$ are direct summands of modules of the form $\Hom_R(K, K^{(\alpha)})\cong \Delta_u(R^{(\alpha)})$ by Lemma~\ref{L:Delta-tf}, thus they are projective objects in the category $\ucontra$. We will show that $\rho:= \Hom_R(K, \pi)$ is a projective cover in $\ucontra$. Take $f\colon \Hom_R(K,D_0) \to \Hom_R(K,D_0)$ such that $\rho f= \rho$. 
By Theorem~\ref{T:cat-equiv}, the adjoint functors $\big((- \otimes_R K), \Hom_R(K, -)\big)$ form equivalences between the subcategories of $\G$-torsion $\G$-divisible modules $\E_\G \cap \D_\G$ and $\G$-torsion-free $u$-contramodules $\F_\G \cap \ucontra$. Thus in particular the functor $\Hom_R(K,-)$ restricted to the subcategories $\E_\G \cap \D_\G \to \F_\G \cap \ucontra$ is full so there exists a $g\colon D_0 \to D_0$ such that $\Hom_R(K, g) = f$.
Thus as $ \pi g=\pi$ implies that $g$ is an automorphism, we conclude that also $f$ is an automorphism, as required.
\end{proof}

\begin{prop}\label{P:R-J-Bass}
Suppose $(\A, \D_\G)$ is a $1$-tilting cotorsion pair over a commutative ring where $\A$ is covering. If $F$ is a Bass $R$-module, then $R/J \otimes_R F$ has a $\clP_0(R/J)$-cover, for every $J\in \G$.
\end{prop}

\begin{proof}
By the proof of Lemma~\ref{L:p-cov-u-contra}, \begin{equation}\label{E:three}
0 \to \Hom_R(K,D_1) \to \Hom_R(K,D_0) \overset{\rho}\to \Hom_R(K,F \otimes_R K) \to 0
\end{equation} is a projective cover of $\Hom_R(K,F \otimes_R K)$ in the category of $u$-contramodules. 
Note that a flat $R$-module is $\G$-torsion-free since $\G$ is faithful and the $\G$-torsion-free class is closed under direct limits. 

By Lemma~\ref{L:Delta-tf} $\Hom_R(K,F \otimes_R K)\cong \Delta_u(F)$ and by Lemma~\ref{L:R-J-tens-tf} $F /JF \cong R/J\otimes_R \Delta_u(F)$. Invoking Lemma~\ref{L:Tor-R/J-Gtf} we also get that $\Tor^R_1(R/J, F) \cong \Tor^R_1(R/J, \Delta_u(F))$, for every $J \in \G$.

Thus applying the functor $(R/J\otimes_R -)$ to the sequence (\ref{E:three}) and using Corollary~\ref{C:R/J-tors-free} we obtain the exact sequence
\begin{equation}\label{E:F-FJ-cov}
0 \to R/J  \otimes_R\Hom_R(K,D_1)\to R/J  \otimes_R\Hom_R(K,D_0)\overset{ \id_{R/J}\otimes_R \rho }\to F/JF \to 0
\end{equation}
 We show that (\ref{E:F-FJ-cov}) is a projective cover of $F/JF$ in $\ModR/J$.

Applying Lemma~\ref{L:Delta-tf} to a free module $R^{(\alpha)}$ we get $\Hom_R(K, K^{(\alpha)})\cong \Delta_u(R^{(\alpha)})$
and by Lemma~\ref{L:R-J-tens-tf}, $R/J\otimes_R\Hom_R(K, K^{(\alpha)})\cong (R/J)^{(\alpha)}.$

 Additionally, since the modules $D_0, D_1$ are summands of a direct sum of copies of $K$, Corollary~\ref{C:R/J-tors-free} implies that $ R/J\otimes_R\Hom_R(K,D_i) $ is an $R/J$-projective module for $i=0,1$.

Furthermore, from the projective cover ~\ref{E:three}, we know that $ \Hom_R(K,D_1)$ is a
superfluous subobject of $\Hom_R(K,D_0)$ in $\ucontra$.

We note that $R/J\otimes_R \Hom_R(K,D_1)$ is a superfluous $R/J$-submodule
of $R/J\otimes_R \Hom_R(K,D_0)$.
In fact, $R/J$ is a $u$-contramodule, for any $J\in \G$ and the image of a superfluous subobject under any morphism in the category is a superfluous subobject. Thus,
$R/J\otimes_R \Hom_R(K,D_1)$ is a superfluous subobject of  $R/J\otimes_R \Hom_R(K,D_0)$ in $\ucontra$. Finally, any submodule of
an $R/J$-module is also a $u$-contramodule, hence $R/J\otimes_R \Hom_R(K,D_1)$ is superfluous in $R/J\otimes_R \Hom_R(K,D_0)$ as an $R/J$-submodule.

Thus we  conclude that (\ref{E:F-FJ-cov}) is a $\clP_0(R/J)$-cover of $F/JF$.
\end{proof}

\begin{thm}\label{T:cov-implies-Gperf}
Suppose $R$ is a commutative ring and $(\A, \D_\G)$ a $1$-tilting cotorsion pair. If $\A$ is covering, then $\psi_R\colon R \to R_\G$ is a perfect localisation, $\pdim R_\G \leq 1$, and $R$ is $\G$-almost perfect.
\end{thm}
\begin{proof}
That $\psi_R\colon R \to R_\G$ is a perfect localisation and $\pdim R_\G \leq 1$ are by Lemma~\ref{L:cov-G-perf} and Proposition~\ref{P:R_G-tilting}. That $R$ is $\G$-almost perfect is by Proposition~\ref{P:Rg-perfect-ring} and Proposition~\ref{P:R-J-Bass}.
\end{proof}
\section{$\clH$-h-local rings}\label{S:h-local}
This section concerns a class of rings which includes the commutative local rings and the h-local rings. We will be looking at $\clH$-h-local rings with respect to a linear topology $\clH$ on a commutative ring $R$. The main result of this section is that the $\clH$-h-local rings can be characterised by the properties of the $\clH$-discrete modules, as will be shown in Proposition~\ref{P:tors-decomp}.

 For a commutative ring $R$, we let $\Max R$ denote the set of all the maximal ideals of $R$.

We will formulate in our setting the results from \cite[Section 4]{BP1}, which were proved in the case of a localisation of a ring at a multiplicative subset. All the proofs can be extended easily to the case of a linear topology $\clH$ on a commutative ring $R$.

\begin{defn}
A commutative ring $R$ is \emph{$\clH$-h-local} if for every open ideal $J \in \clH$, $J$ is contained only in finitely many maximal ideals of $R$ and every open prime ideal in $\clH$ is contained in only one maximal ideal.

A ring commutative $R$ is \emph{$\clH$-h-nil} if every element $J \in \clH$ is contained only in finitely many maximal ideals of $R$ and every prime ideal of $R$ in $\clH$ is maximal.
\end{defn}
It is clear that every $\clH$-h-nil ring is $\clH$-h-local. We first give a sufficient condition for a ring to be $\clH$-h-nil. 
\begin{lem}\label{L:R-J-perf-h-nil}
Let $\clH$ be a linear topology on a commutative $R$. If $R/J$ is perfect for every $J \in \clH$, then $R$ is $\clH$-h-nil.
\end{lem}
\begin{proof}
By Proposition~\ref{P:perfect}, $R/J$ has only finitely many maximal ideals.

Take a prime $\p \in \clH$. Then $R/\p$ is a perfect domain, so is a field (by the final statement of Proposition~\ref{P:perfect}), so it follows that $\p$ must be maximal.
\end{proof}
 Recall that for every right linear topology $\clH$ on a ring $R$ the class of \emph{$\clH$-discrete} modules consist of $\{ M\mid \Ann x \in \clH \text{ for all } x \in M \}$. 

The following holds for any linear topology on a commutative ring, and is our generalisation of \cite[Lemma 4.2]{BP1}. 
\begin{lem} \label{L:RmRn-div}
Let $\clH$ be a linear topology on a commutative ring $R$ such that every prime in $\clH$ is contained in only one maximal ideal. Then for maximal ideals $\m \neq \n$ of $R$, 
and for each $\clH$-discrete module $M$, $M \otimes_R R_\m \otimes_R R_\n=0$.
\end{lem}
\begin{proof}
Let $\phi\colon R \to R_\m \otimes_R R_\n$ denote the localisation map. We will first show that the statement holds for $R/J$ for a fixed $J \in \clH$. Take $\q$ a prime ideal in $R_\m \otimes_R R_\n$. Then there is a unique prime $\p$ of $R$ such that $\p \subseteq \m \cap \n$ and $\q = \p(R_\m \otimes_R R_\n)$. By assumption, $\p \notin \clH$ as it is a prime contained in two maximal ideals. Therefore, $J \nsubseteq \p$ so $J R_\p = R_\p$. 

We will show that for every prime ideal $\q$ of $R_\m \otimes_R R_\n$, the localisation of $R/J\otimes_RR_\m \otimes_R R_\n$ at $\q$ is zero. Fix a prime $\q$ of $R_\m \otimes_R R_\n$ and let $\p=\phi^{-1}(\q)$. Then $R_\p \cong (R_\m \otimes_R R_\n)_\q$ as $R$-modules. Moreover, as we know that $R/J \otimes_R R_\p=0$ by the argument in the first paragraph, we conclude the first statement of the lemma.

The statement now follows easily as every $\clH$-discrete module $N$ is an epimorphic image of modules of the form $\bigoplus_\alpha R/J_\alpha$ with $J_\alpha \in \clH$. 
\[
0= (\bigoplus_\alpha R/J_\alpha) \otimes_R R_\m \otimes_R R_\n \to N \otimes_R R_\m \otimes_R R_\n \to 0
\] 
\end{proof}
The following two propositions are the main results of this section, which generalise \cite[Proposition 4.3 and Lemma 4.4]{BP1}. For the latter, we don't include a proof as it follows analogously from the original proof using Lemma~\ref{L:RmRn-div}, our version of \cite[Lemma 4.2]{BP1}.
\begin{prop}\cite[Proposition 4.3]{BP1}\label{P:tors-decomp}
Suppose $\clH$ is a linear topology over a commutative ring $R$. The following are equivalent. 
\begin{enumerate}
\item[(i)] $R$ is $\clH$-h-local.
\item[(ii)] $N \cong \underset{{\m \in \Max R}}\bigoplus_{} N_\m$ for every $\clH$-discrete module $N$.
\item[(iii)] $N \cong \underset{\substack{\m \in \clH \\ \m \in \Max R}}\bigoplus N_\m$ for every $\clH$-discrete module $N$.
\end{enumerate}
Moreover, the above conditions hold when $R/J$ is a perfect ring for every $J \in \clH$.
\end{prop}
\begin{proof}
(i) $\Rightarrow$ (ii). We begin by showing that statement (ii) holds for the cyclic modules $R/J$ with $J \in \clH$. By assumption, $J$ is contained in finitely many maximal ideals so in $R/J \to R/J\otimes_RR_\m$, $1+J$ is mapped to a non-zero element of $R/J\otimes_R R_\m$ for only finitely many maximal ideals. Thus there is the following natural monomorphism. 
\[\xymatrix@R=5pt@C=7pt{
\Psi_{R/J}\colon R/J \ar@{^{(}->}[rr]&& \underset{\m \in \Max R }\bigoplus (R/J)_\m &\subseteq \underset{\m \in \Max R }\prod (R/J)_\m \\
\hspace{37pt} r+J \ar@{|->}[rr]&& \underset{\m \in \Max R } \sum(r+J)_\m}
\]
We will show that $\Psi_{R/J}$ is surjective by showing that for every maximal ideal $\n$ of $R$, the localisations $\big(\Psi_{R/J}(R/J)\big)_\n$ and $\big(\underset{\m \in \Max R }\bigoplus (R/J)_\m\big)_\n$ coincide. To begin, if $\n \notin \clH$ is maximal, then for each $J \in \G$, $(R/J)_\n=0$ as there exists an $a \in J \setminus \n$, and it also follows that $ \big ( \bigoplus_{\m \in \Max R } (R/J)_\m \big)_\n=0$. For a maximal ideal $\n \in \clH$, by Lemma~\ref{L:RmRn-div}, $(R/J)_\m \otimes_R R_\n =0$ for $\m \neq \n$. So clearly $\big( \bigoplus_{\m \in \Max R } (R/J)_\m \big)_\n=(R/J)_\n =\Psi_{R/J}(R/J)_\n$, where $(R/J)_\n$ is a submodule of $ \bigoplus_{\m \in \Max R } (R/J)_\m $ so we are done. 

For a $\clH$-discrete module $N$, consider a short exact sequence of the following form where $J _\alpha \in \clH$ and all the modules are $\clH$-discrete as the class of $\clH$-discrete modules is closed under submodules and quotients (that is, it is hereditary pretorsion). 
\begin{equation}\label{E:hello}
0 \to H \to \bigoplus_{\substack{\alpha }} R/J_\alpha \to N \to 0
\end{equation}
Consider the following commuting diagram formed by taking the direct sum of all $\bigoplus_{\substack{\m \in \Max R}}(R_\m \otimes_R -)$ applied to (\ref{E:hello}), and $\psi_H, \psi_N$ the natural maps sending each element to its image in the localisations, which can be seen to be well defined (that is, contained in the direct sum) considering the isomorphism for each $R/J$.
\begin{equation}\label{E:bottle}
\xymatrix{
0 \ar[r] & H \ar[r] \ar[d]^{\psi_H} & \bigoplus_{\substack{\alpha }} R/J_\alpha \ar[r] \ar[d]^\cong & N \ar[r] \ar[d]^{\psi_N} & 0\\
0 \ar[r] & \underset{\m \in \Max R}\bigoplus H_\m \ar[r] & \underset{\alpha}\bigoplus \big( \underset{\m \in \Max R }\bigoplus (R/J_\alpha)_{\m} \big) \ar[r] & \underset{\m \in \Max R} \bigoplus N_\m \ar[r] & 0}
\end{equation}
Thus $\psi_N$ is surjective by the snake lemma applied to (\ref{E:bottle}). Additionally, as also $H$ is $\clH$-discrete, the same argument says that $\psi_H$ is surjective. Thus $\psi_N$ must be an isomorphism again by the snake lemma applied to (\ref{E:bottle}). 

(ii) $\Rightarrow$ (iii). If $J \in \clH$ and $\n$ is a maximal ideal of $R$ not contained in $\clH$ then clearly $J\nsubseteq \n$, hence $(R/J)_\n=0$. Therefore, using that every $\clH$-discrete module $M$ is the image of cyclic $\clH$-discrete modules, $M_\n=0$ for every maximal $\n \notin \clH$.

(iii) $\Rightarrow$ (i). By assumption $R/J \cong \bigoplus_{\m \in \Max R \cap \clH } (R/J)_\m$. The direct sum must be finite as $R/J$ is cyclic. Moreover, if $(R/J)_\n=0$ for $\n$ maximal then $J \nsubseteq \n$. This shows that $J$ is contained in only finitely many maximal ideals. To see that every prime $\p$ of $\clH$ must be contained only in one maximal ideal, suppose $\p\subseteq \m \cap \n$ where $\m \neq \n$ are maximal and consider $R/\p \cong \bigoplus_{\m \in \Max R\cap \clH } (R/\p)_\m$. For every $\p \subseteq \m$, $(R/\p)_\m\otimes_RR_\p\cong R_\p/\p R_\p$ and $R_\p/\p R_\p$ cannot contain two direct sum copies of itself, since it is a field.
\end{proof}

\begin{prop}\cite[Lemma 4.4]{BP1}\label{P:mor-sum-loc}
Let $R$ be a $\clH$-h-local ring. Let $\{M(\m)\}_\m \in \Max R$, $\{N(\m)\}_\m \in \Max R$ be two collections of modules such that $M(\m), N(\m)$ are $R_\m$-modules for each maximal ideal $\m$ of $R$. Suppose the modules $\{M(\m)\}$ are $\clH$-discrete. Then any morphism $\bigoplus_\m M(\m) \to \bigoplus_\m N(\m)$ is a direct sum of $R_\m$-module homomorphisms $M(\m) \to N(\m)$.
\end{prop}

\section{When $R$ is a $\G$-almost perfect ring}\label{S:Gperf}
In this section we assume that $(\A, \D_\G)$ is a $1$-tilting cotorsion pair with associated Gabriel topology $\G$ and that $R$ is $\G$-almost perfect (that is $R_\G$ is a perfect ring and $R/J$ is a perfect ring for every $J \in \G$).

The purpose of this section is to show that under these assumptions $\A$ is covering, that is a sort of converse to Theorem~\ref{T:cov-implies-Gperf}.

To prove the next lemma, we recall the following construction. Suppose $M$ is a finitely presented right $R$-module with projective presentation $P_1 \overset{\rho} \to P_0 \to M \to 0$ where $P_0, P_1$ are finitely generated projective modules. The \emph{transpose} of $M$, denoted $\Tr (M)$, is the cokernel of the map $\rho^\ast\colon P^\ast_0 \to P^\ast_1$ where $(-)^\ast := \Hom_R(-, R)$.

If $\G$ is a faithful Gabriel topology, then for every finitely generated ideal $J \in \G$, $0\to R\to R^n\to\Tr (R/J)\to 0$ is a projective resolution of $\Tr (R/J)$ where $n$ is the number of generators of $J$.
\begin{lem}\label{L:fdimRg-0}
Let $R$ be a commutative ring. Suppose $(\A, \D_\G)$ is a $1$-tilting cotorsion pair, $\G$ the associated Gabriel topology and $\fdim R_\G=0$. Then $\G$ arises from a perfect localisation, or equivalently $R_\G$ is $\G$-divisible.

In particular, if $R_\G$ is a perfect ring, then the statement holds.
\end{lem}
\begin{proof}
For each finitely generated $J \in \G$, \cite[Lemma 3.3]{H} shows that $R_\G \otimes_R R/J \cong \Ext^1_R( \Tr R/J, R_\G)$ and as $\pdim \Tr R/J \leq 1$, Lemma~\ref{L:G-top-facts}(ii) yields $\Tor^R_1(\Tr R/J, R_\G)=0$. Thus applying $( R_\G \otimes_R -)$ to a projective resolution of $\Tr R/J$, we get the following.
\[
0 \to R_\G \to R_\G^n \to R_\G \otimes_R \Tr R/J \to 0
\]
By assumption $\fdim R_\G=0$ so $R_\G \otimes_R \Tr R/J $ is $R_\G$-projective. Next consider the following isomorphism which follows as $\Tor^R_i(\Tr R/J, R_\G)=0$ for $i >0$.
\[
\Ext^1_R(\Tr R/J, R_\G) \cong \Ext^1_{R_\G}(R_\G \otimes_R \Tr R/J , R_\G)=0
\] 
The last module vanishes as $R_\G \otimes_R \Tr R/J $ is $R_\G$-projective, so $R/J \otimes_R R_\G \cong\Ext^1_R(\Tr R/J, R_\G)=0$ for each $J \in \G$ hence $R_\G$ is $\G$-divisible.

If $R_\G$ is a perfect ring, then by Proposition~\ref{P:perfect}, $\Fdim R_\G=0$, so the statement applies.
\end{proof}


\begin{rem}\label{R:pos-ref}\emph{
It was pointed out by Leonid Positselski (via private correspondence) that if $\G$ is a perfect Gabriel topology on a ring $R$ such that the rings $R/J$ are perfect for every $J \in \G$, it follows that $\pdim R_\G \leq 1$. His proof is a generalisation of \cite[Theorem 6.13]{BP1}.}
\end{rem}
The above remark with Lemma~\ref{L:fdimRg-0} allows us to state the following.
\begin{prop}\label{P:G-AP} If $(\A, \D_\G)$ is a $1$-tilting cotorsion pair with associated Gabriel topology $\G$ such that $R$ is $\G$-almost perfect, then $\G$ is a perfect localisation, $\pdim R_\G\leq 1$ and $R_\G\oplus R_\G/R$ is a corresponding $1$-tilting module.
\end{prop}

Thus with Proposition~\ref{P:G-AP}, we can consider the following setting:
\begin{set}\label{set:pure-split}
We assume that $(\A, \D_\G)$ is a $1$-tilting cotorsion pair arising from a flat injective ring epimorphism $u\colon R \to U$ such that $\pdim U\leq1$ and $\Gen(U) = \D_\G$ as in the equivalent statements of Proposition~\ref{P:Hrb-5.4}, so that $U\oplus K$ ($K:=U/R$) is the associated $1$-tilting module.
\end{set}
Thus, if $R$ is moreover $\G$-almost perfect, to show that $\A$ is covering it is sufficient to show that $U\oplus K$ is $\Sigma$-pure-split, as then $\A$ is closed under direct limits using Proposition~\ref{P:indlim-sigmapuresplit}. To show that $U\oplus K$ is $\Sigma$-pure-split, the problem naturally divides into two parts: showing that each of $U$ and $K$ are $\Sigma$-pure-split.

\subsection{If $K$ is $\Sigma$-pure split then $U\oplus K$ is $\Sigma$-pure split} \label{SS:pure-Gperf}
Consider a pure exact sequence:\begin{equation} \label{E:pure}
0 \to X \to T \to Y \to 0
\end{equation}
where $T \in \Add (U \oplus K)$.

Then $X, Y \in \D_\G$ as $T \in \D_\G$ and the tilting class is closed under pure submodules, so the sequence vanishes when one applies $(R/J \otimes_R -)$. 

\begin{lem}\label{L:tf-splits}
Let $R$ be as in Setting~\ref{set:pure-split} such that $U$ is a perfect ring. Applying the functor $(- \otimes_R U)$ to sequence (\ref{E:pure}) we find a split exact sequence of projective $U$-modules:
\begin{equation} \label{E:pure-tens-U}
0 \to X \otimes_R U \to T\otimes_R U \to Y \otimes_R U \to 0
\end{equation}
\end{lem}

\begin{proof}
The sequence (\ref{E:pure-tens-U}) is an exact sequence in $\ModU$ and it is also pure since (\ref{E:pure}) is pure. Moreover, as $T \in \Add(U \oplus K)$, $T\otimes_RU \in \Add(U)$ and thus it is $U$-projective.

Thus $Y \otimes_R U$ is a flat $U$-module and therefore it is $U$-projective as $U$ is a perfect ring. So the sequence splits in $\ModU$ and hence in $\ModR$. Also note that this implies that $X \otimes_R U$ is flat in $\ModR$.
\end{proof}

From now on $t(M)$ will denote the torsion submodule of a module $M$ with respect to the $\G$-torsion class $\E_\G$.
\begin{lem}\label{L:pure-t-tf}
Let $R$ be as in Setting~\ref{set:pure-split} such that $U$ is a perfect ring and let $X, T, Y$ be as in (\ref{E:pure}). Then
\[
0 \to t(X) \to t(T) \to t(Y) \to 0
\]
is a pure exact sequence.
\end{lem}
\begin{proof}
We claim diagram (\ref{E:square}) has exact rows and exact columns. This is because the bottom row is exact as (\ref{E:pure}) is pure exact, and by the snake lemma and the fact that $X \otimes_R K =0$ as $X$ is $\G$-divisible, forces the top row to be exact.
\begin{equation}\label{E:square}
\xymatrix{
& 0 \ar[d] & 0 \ar[d] & 0 \ar[d] &\\
0 \ar[r] &t(X) \ar[r] \ar[d] & t(T) \ar[r] \ar[d] & t(Y) \ar[r] \ar[d] & 0\\
0 \ar[r] &X \ar[r] \ar[d] & T \ar[r] \ar[d] & Y \ar[r] \ar[d] & 0\\
0 \ar[r] &X \otimes_R U \ar[r] \ar[d] & T\otimes_R U \ar[r] \ar[d] & Y \otimes_R U \ar[r] \ar[d] & 0\\
& 0 & 0 & 0 &}
\end{equation}
To show that the top row is pure exact it is enough to show that, for every $N\in \ModR$, the connection map $\delta\colon \Tor^R_1(t(Y), N) \to t(X) \otimes_R N$ is zero. By Lemma~\ref{L:tf-splits} $X \otimes_R U$ is a flat $U$-module, hence also flat as an $R$-module, thus $\Tor^R_1(X \otimes_R U, N)=0$. We want to show that $\delta =0$. Applying $(- \otimes_R N)$ to the diagram above we obtain:
\[
\xymatrix{
&\Tor^R_1(X \otimes_R U, N)=0 \ar[d]\\
\Tor^R_1(t(Y), N) \ar[d] \ar[r]^\delta & t(X) \otimes_R N \ar[d]^\varepsilon\\
\Tor^R_1(Y, N) \ar[r]^0 & X \otimes_R N}
\]
So $\varepsilon \delta =0$ and as $\varepsilon$ is a monomorphism, $\delta=0$ as required.
\end{proof}

\begin{lem}\label{L:if-t-tf-split}
Let $R$ be as in Setting~\ref{set:pure-split} such that $U$ is a perfect ring.
Suppose that $K$ is $\Sigma$-pure split. Then $U \oplus K$ is $\Sigma$-pure split, that is every pure embedding as in (\ref{E:pure}) splits.
\end{lem}
\begin{proof}
Consider a pure exact sequence as in (\ref{E:pure}). We show that in our setting the sequence splits. \\ By our assumption the sequence $0 \to t(X) \to t(T) \to t(Y) \to 0
$ splits, since by Lemma~\ref{L:pure-t-tf} it is pure exact and $t(T)\in \Add K$. So $t(Y) \in \Add K$. Moreover, by Lemma~\ref{L:tf-splits}, $Y\otimes_RU \in \Add(U)$. Since $K\in U^\perp$ the sequence:
\[
0 \to t(Y) \to Y \to Y \otimes_R U \to 0
\]
splits. Thus the sequence (\ref{E:pure}) splits as $X \in \D_\G$.
\end{proof}

Thus, our next aim is to show that in Setting~\ref{set:pure-split} when $R/J$ is perfect for each $J \in \G$, $K$ is $\Sigma$-pure split.
To this end, consider a pure exact sequence
\begin{equation} \label{E:pure-t}
0 \to \overline X \to \overline T \to \overline Y \to 0
\end{equation}
with $\overline T \in \Add(K)$. 
\begin{facts}\label{F:Posit} \emph{The terms in sequence (\ref{E:pure-t}) are $\G$-torsion and $\G$-divisible modules. Hence we can use the category equivalence of Theorem~\ref{T:cat-equiv} between the subcategories of $\G$-torsion $\G$-divisible modules and the $\G$-torsion-free $u$-contramodules via the adjoint functors $\big((-\otimes_R K), \Hom_R(K, -)\big)$. We will show that $\Hom_R(K,\overline Y)$ is a projective object in $\ucontra$ and moreover that the sequence splits in the category of $\G$-torsion-free $u$-contramodules. Thus also the original sequence (\ref{E:pure-t}) splits in the category of $\G$-torsion $\G$-divisible modules.}

\emph{Moreover, we will use that for a $\G$-torsion-free module $N$, (in particular a free module $R^{(\beta)}$), Lemma~\ref{L:Delta-tf} gives an isomorphism $\mu_N: \Hom_R(K, K \otimes_RN)\cong \Delta_u(N)$ and these are $u$-contramodules by Lemma~\ref{L:some-ucontras}. Also we use regularly Lemma~\ref{L:R-J-tens-tf}, that is $M/JM \cong R/J\otimes_R\Delta_u(M)$ for any $R$-module $M$ and every $J \in \G$. Finally, we also recall that with the assumption $\pdim U \leq 1$, $\ucontra$ is an abelian category and the direct summands of modules of the form $\Delta_u(R^{(\beta)})$ for some cardinal $\beta$ are the projective objects in $\ucontra$ as stated in Proposition~\ref{P:Delta-left-adjoint}.} 
\end{facts}

Before reducing to the local case, we want to make some remarks and results about the module $\Hom_R(K, \overline Y)$.
\begin{rem}\label{R:L-HomKY}\emph{
The aim of the next results will be to show that the $\G$-torsion-free $u$-contramodule $\Hom_R(K, \overline Y)$ with $\overline Y$ from the sequence (\ref{E:pure-t}) is a projective object in $\ucontra$, however most of the results can be generalised to a $\G$-torsion-free $u$-contramodule $M$ such that $M \otimes_R K \in \F_1(R)$. }

\emph{The module $\Hom_R(K, \overline Y)$ satisfies the assumptions on $M$ as $\overline Y \cong \Hom_R(K, \overline Y)\otimes_R K$ by Theorem~\ref{T:cat-equiv} since $\overline Y$ is $\G$-divisible and $\G$-torsion. Furthermore, $\Hom_R(K, \overline Y)\otimes_R K \in \F_1(R)$ as $\overline X, \overline Y \in \F_1(R)$ since $\overline T \in \F_1(R)$ and the sequence (\ref{E:pure-t}) is pure-exact.}

\end{rem}

\begin{lem}\label{L:Tor-ucontra-torsion}
Let $R$ be a ring as in Setting~\ref{set:pure-split} such that $R/J$ is a perfect ring for each $J \in \G$. Suppose $M$ is a $\G$-torsion-free $u$-contramodule such that $M \otimes_R K \in \F_1(R)$ and $L$ is a $\G$-torsion module. Then $\Tor^R_1(L, M)=0$. 
\end{lem}
\begin{proof}
Fix $M$ and $L$ as in the assumptions and consider the following exact sequence.
\[
0 \to M \to M \otimes_R U \to M\otimes_R K \to 0
\]
Apply $(L \otimes_R -)$ to get the exact sequence 
\[
0 = \Tor^R_2(L, M \otimes_R K) \to \Tor^R_1(L , M) \to \Tor^R_1(L, M \otimes_R U )
\]
Thus as $U$ is flat, we have that $\Tor^R_1(L, M \otimes_R U ) \cong \Tor^U_1( L \otimes_R U , M \otimes_R U)$ which is zero as $L \otimes_R U=0$ since $L$ is $\G$-torsion.
\end{proof}

\subsection{When $R$ is local and $R/J$ is a perfect ring for each $J \in \G$ }\label{SS:local-Gperf}
 In this subsection we will assume that $R$ is a local ring with maximal ideal $\m$.\\

 In the rest of this section we will show that $\Hom_R(K, \overline Y)$ (or a $\G$-torsion-free $u$-contramodule $M$ such that $M \otimes_R K \in \F_1(R)$, see Remark~\ref{R:L-HomKY}) is a projective object in $\ucontra$ using the method of Positselski in \cite[Lemma 8.2 and Theorem 8.3]{Pos5}, although in a much simpler setting. 
 
\begin{rem}\label{R:R-J}{\em
By assumption $R$ is local, so $R/J$ is a perfect local ring for each $J \in \G$. Therefore the maximal ideal of each $R/J$ is T-nilpotent, so by \cite[Lemma 28.3]{AF} it follows that $M \m \ll M$ for every $R/J$-module $M$. Moreover by Proposition~\ref{P:perfect} (iv), every $R/J$-module has a non-zero socle.}
\end{rem}

\begin{lem} \label{L:div-or-sup}
Let $R$ be a commutative local ring with a faithful finitely generated perfect Gabriel topology $\G$ such that $R/J$ is a perfect ring for each $J \in \G$.
Then every non-zero $R$-module $M$ is either in $\D_\G$ or $M \m \neq M$.
\end{lem}
\begin{proof}
Suppose $M$ is not in $\D_\G$. Then there exists $J \in \G$ such that $M/M J \neq 0$. By Remark~\ref{R:R-J} we have the following strict inclusion. 
 \[
(M/MJ) \m = (M\m)/(MJ) \subsetneq M/M J
\]
So it follows that $ M \m \subsetneq M$, as required. 
\end{proof}

\begin{prop}\label{P:f-epi}
Let $R$ be a commutative local ring with a faithful finitely generated perfect Gabriel topology $\G$ such that $R/J$ is a perfect ring for each $J \in \G$.
Let $M$ be a $u$-contramodule. Then there is a cardinal $\beta$ and an epimorphism $f$ that makes the following diagram commute. 
\begin{equation}\label{E:pencil}
 \xymatrix{
&& \Delta_u(R^{(\beta)}) \ar[d]^f \ar[r]^p &(R/\m)^{(\beta) \ar[d]^\cong} \ar[r] &0\\
0 \ar[r] & M \m \ar[r] & M \ar[r] \ar[r]& M/M \m \ar[r] & 0}
\end{equation}
\end{prop}
\begin{proof} Consider the exact sequence $
0 \to M \m \to M \to M/ M\m \to 0
.$

As $M/M \m$ is an $R/\m$-module, there exists a cardinal $\beta$ such that $(R/\m) ^{(\beta)} \cong M/M \m$. 
Let $p\colon \Delta_u(R^{(\beta)}) \to (R/\m)^{(\beta)}$ be the composition of the the natural projection map  $\Delta_u(R^{(\beta)})\to \Delta_u(R^{(\beta)})/ \Delta_u(R^{(\beta)})\m$ with the isomorphism  $\Delta_u(R^{(\beta)})/ \Delta_u(R^{(\beta)})\m \cong (R/\m)^{(\beta)}$ guaranteed by Lemma~\ref{L:R-J-tens-tf}.

Consider the diagram
\begin{equation}\label{E:pencil}
 \xymatrix{
&& \Delta_u(R^{(\beta)}) \ar[d]^f \ar[r]^p &(R/\m)^{(\beta) \ar[d]^\cong} \ar[r] &0\\
0 \ar[r] & M \m \ar[r] & M \ar[r] \ar[r]^{p_M}& M/M \m \ar[r] & 0}
\end{equation}
where $f$ exists since all the modules in the above diagram are $u$-contramodules and $\Delta_u(R^{(\beta)})$ is a projective object in $\ucontra$. To see that $f$ is an epimorphism, note that as $\Delta_u(R^{(\beta)}) \to M/M \m$ is an epimorphism, it follows that $\Img f + M \m = M$. By \[
(M / \Img f)\m = (M \m + \Img f)/\Img f = M/ \Img f
\] and Lemma~\ref{L:div-or-sup}, it follows that $M/ \Img f$ is $\G$-divisible.
However, $f$ is a map of $u$-contramodules, so also $\Coker f = M / \Img f$ is a $u$-contramodule, thus $M / \Img f $ contains no non-zero $\G$-divisible submodule. We conclude that $M/\Img f =0$, so $f$ is an epimorphism as required. 
\end{proof}
The following proposition uses the results from Section~\ref{S:top}. 
\begin{prop}\label{P:f-tensor-iso}
Let $R$ be a commutative local ring with a faithful finitely generated perfect Gabriel topology $\G$ such that $R/J$ is a perfect ring for each $J \in \G$.
 Then for every $\G$-torsion-free $u$-contramodule $M$ such that $M \otimes_RK \in \F_1(R)$ the morphism $f$ as in (\ref{E:pencil})
is an isomorphism.

In particular, $M$ is a projective object in $\ucontra$.
\end{prop}
\begin{proof}
Let $\beta$ and $f$ be as in Proposition~\ref{P:f-epi}. 

For every  $J \in \G$, $R/J$ is a perfect local ring, hence $R/J$ is a semiartinian module. Consider a Loewy series $\{J_{\sigma}/J\}_{\sigma<\tau}$ of $R/J$, that is $ J_{\sigma+1}/J_{\sigma}\cong R/\m$ for every $\sigma<\tau$  and $R/J=\bigcup_{\sigma < \tau} J_\sigma/J$.
 
 By Lemma~\ref{L:Tor-R/J-Gtf} and Lemma~\ref{L:Tor-ucontra-torsion} we have $\Tor^R_1(R/\m, \Delta_u(R^{(\beta)}))=0=\Tor^R_1(R/\m, M)$. From the diagram (\ref{E:pencil}) where $f$ is an epimorphism we see that for every ordinal $\sigma$, we have the following commuting diagram.
\[\hspace{-0.3cm}
\xymatrix@C=0.7cm{
&&&0 \ar[d]\\
0 \ar[r] & J_{\sigma}/J \otimes_R \Delta_u(R^{(\beta)}) \ar[r] \ar[d]^{\id_{J_{\sigma}/J} \otimes_R f} & J_{\sigma+1}/J \otimes_R \Delta_u(R^{(\beta)}) \ar[r] \ar[d]^{\id_{J_{\sigma+1}/J} \otimes_R f} & R/\m \otimes_R \Delta_u(R^{(\beta)}) \ar[r] \ar[d]^{\id_{R/\m} \otimes_R f}_\cong & 0\\
0 \ar[r] & J_{\sigma}/J \otimes_R M\ar[r] \ar[d]& J_{\sigma+1}/J \otimes_R M \ar[r] \ar[d] & R/\m \otimes_R M \ar[r] \ar[d]& 0\\
&0&0 &0}
\]
We will first show that $\id_{R/J}\otimes_R f$ is an isomorphism by transfinite induction on $\sigma$. It is clear in the base case of $\sigma =1$. If $\id_{J_{\sigma}/J} \otimes_R f$ is an isomorphism, then by the five-lemma, as the two outer vertical morphisms of the above diagram are isomorphisms, also $\id_{J_{\sigma+1}/J} \otimes_R f$ is an isomorphism. 

Let $\rho<\tau$ be a limit ordinal. By induction, $\id_{J_{\sigma}/J} \otimes_R f$ is an isomorphism for every $\sigma<\rho$. Hence, since the tensor product commutes with direct limits we get  the isomorphism $\id_{J_{\rho}/J} \otimes_R f$. Now:
\begin{align*}
\big( \bigcup_{\sigma < \tau} J_\sigma/J \big) \otimes_R \Delta_u(R^{(\beta)})
&= \bigcup_{\sigma < \tau} \big(J_\sigma/J \otimes_R \Delta_u(R^{(\beta)})\big) \\
& \cong \bigcup_{\sigma < \tau} (J_\sigma/J \otimes_R M) \\
&= \big( \bigcup_{\sigma < \tau} J_\sigma/J \big) \otimes_R M
\end{align*}

As $R/J = \bigcup_\alpha J_\alpha /J$ we have shown that \[\id_{R/J}\otimes_R f\colon R/J\otimes_R\Delta_u(R^{(\beta)}) \cong M/ JM.\] 
Now note that the above isomorphism implies that kernel of $f\colon \Delta_u(R^{(\beta)}) \to M$ is contained in every $\Delta_u(R^{(\beta)})J$, thus $\Ker f \subseteq \bigcap_{J \in \G}\Delta_u(R^{(\beta)})J$. However, as $R^{(\beta)}$ is $\G$-torsion-free we have, by Lemma~\ref{L:Delta-tf}, $\Delta_u(R^{(\beta)})\cong \Hom_R(K,K^{(\beta)})$, which is already $\G$-separated by Lemma~\ref{L:HomK-H-sep}, so $\bigcap_{J \in \G}\Delta_u(R^{(\beta)})J$ vanishes. We conclude that $f$ is an isomorphism.
\end{proof}

\begin{prop}\label{P:local-splits}
Let $R$ be a commutative local ring with a faithful finitely generated perfect Gabriel topology $\G$ such that $R/J$ is a perfect ring for each $J \in \G$. Consider the pure exact sequence with $\overline T \in \Add(K)$. 
\begin{equation}\tag{\ref{E:pure-t}}
0 \to \overline X \to \overline T \to \overline Y \to 0 
\end{equation}
Then the sequence splits. In other words, $K$ is $\Sigma$-pure-split.
\end{prop}
\begin{proof}

By Proposition~\ref{P:f-tensor-iso} and Remark~\ref{R:L-HomKY}, we have that $M = \Hom_R(K,\overline Y)$ is a projective object in $\ucontra$, therefore the following sequence of $u$-contramodules (which is $\Hom_R(K,-)$ applied to (\ref{E:pure-t})) splits.
\[
0 \to \Hom_R(K,\overline X) \to \Hom_R(K,\overline T) \to \Hom_R(K,\overline Y) \to 0 
\]
Applying $(- \otimes_R K)$, we recover the original short exact sequence up to isomorphism, which also splits. 
\end{proof}

\subsection{Final results}\label{SS:results-Gperf}
We have shown that for $R$ a commutative local ring, if $R/J$ is perfect for every $J \in \G$, a pure submodule of a module $\overline T \in \Add(K)$ splits. We will now extend this to the global case in this final subsection.

We recall that since $R/J$ is perfect for each $J \in \G$, by Lemma~\ref{L:R-J-perf-h-nil} the ring $R$ is $\G$-h-nil hence, the equivalent statements of Proposition~\ref{P:tors-decomp} hold. That is, we use in particular that for every $\G$-torsion module $M$, $M \cong \underset{{\m \in \Max R}} \bigoplus M_\m$, where $\m$ runs over all the maximal ideals of $R$.

\begin{prop}\label{P:K-pure-split}
Let $R$ be a commutative ring with a faithful finitely generated perfect Gabriel topology $\G$ with perfect localisation $u\colon R \to U$ such that $R/J$ is a perfect ring for each $J \in \G$. Then $K$ is $\Sigma$-pure-split where $K = U/u(R)$. 
\end{prop}
\begin{proof}
Take $0 \to \overline X \to \overline T \overset{\rho}\to \overline Y \to 0$ a pure exact sequence. By Proposition~\ref{P:tors-decomp}, $\overline T = \bigoplus_\m (\overline T)_\m$ and $\overline Y = \bigoplus_\m (\overline Y)_\m$. Additionally by Proposition~\ref{P:mor-sum-loc}, the morphism $\rho$ is a direct sum of surjective maps $(\overline T)_\m \to (\overline Y)_\m$ and is also a pure epimorphism. By Proposition~\ref{P:local-splits}, each $(\overline Y)_\m$ is in $\Add(K)_\m$, thus also $\overline Y \in \Add(K)$. Thus $\rho$ is a split epimorphism as $\overline X \in \D_\G$.
\end{proof}

\begin{thm} \label{T:tilting-cover}
Let $R$ be a commutative ring and $(\A, \D_\G)$ a $1$-tilting cotorsion pair with associated Gabriel topology $\G$ such that $R$ is $\G$-almost perfect. Then $R_\G \oplus R_\G/R$ is an associated $1$-tilting module and $R_\G \oplus R_\G/R$ is $\Sigma$-pure-split, so $\A$ is closed under direct limits.
\end{thm}
\begin{proof}
By Lemma~\ref{L:fdimRg-0}, $R_\G$ is $\G$-divisible, so $\G$ is a perfect Gabriel topology. Next, if the $R/J$ are perfect rings for $J \in \G$ and $\G$ is a perfect Gabriel topology, it follows that $\pdim R_\G \leq 1$ by Remark~\ref{R:pos-ref}.

That $R_\G \oplus R_\G/R$ is $\Sigma$-pure-split is a combination of Lemma~\ref{L:if-t-tf-split} and Proposition~\ref{P:K-pure-split}. Finally by Proposition~\ref{P:indlim-sigmapuresplit} we conclude that $\A$ is closed under direct limits.
\end{proof}
 The following definition has been introduced in \cite{Pos5} (see also \cite{BP4}).
\begin{defn} A linearly topological ring is \emph{pro-perfect} if it is separated, complete, and has a basis of neighbourhoods of zero formed by two-sided ideals such that all of its discrete quotient rings are perfect. 
\end{defn}
Finally combining the above theorem with the results in Section~\ref{S:Gperf} and Section~\ref{S:cov-G-perf} we obtain the main result of this paper. 

\begin{thm}\label{T:characterisation-cov}
Suppose $(\A, \D_\G)$ is a $1$-tilting cotorsion pair, $\G$ the associated Gabriel topology and $\mathfrak R$ the topological ring $\End_R(K)$. The following are equivalent.
\begin{enumerate}
\item[(i)] $\A$ is closed under direct limits.
\item[(ii)] $\A$ is covering.
\item[(iii)] $R$ is $\G$-almost perfect.
\item[(iv)] $R_\G$ is a perfect ring, and $\mathfrak{R}$ is pro-perfect.
\end{enumerate} 
Moreover, if these equivalent conditions hold $R \to R_\G$ is a perfect localisation and $\pdim R_\G \leq 1$.
\end{thm}
\begin{proof}
(i) $\Rightarrow$ (ii) is \cite[Theorem 2.2.8]{Xu} and \cite[Theorem 6.11]{GT12}. 

(ii) $\Rightarrow$ (iii) is Theorem~\ref{T:cov-implies-Gperf}. 

(iii) $\Rightarrow$ (i) is Theorem~\ref{T:tilting-cover}.

(iii) $\Leftrightarrow$ (iv) In both statements, $R_\G$ is perfect, so by Lemma~\ref{L:fdimRg-0}, $\G$ is a perfect localisation. Hence by Proposition~\ref{P:same-topologies}, $\mathfrak{R}$ is closed and separated with respect to the $\G$-topology. Also by Lemma~\ref{L:Delta-tf} and Lemma~\ref{L:R-J-tens-tf}, $\mathfrak{R}/\mathfrak{R}J \cong R/J$, so the discrete quotient rings of $\mathfrak{R}$ are perfect if and only if the $R/J$ are perfect for each $J \in \G$.

The final statements follow by Theorem~\ref{T:tilting-cover}.
\end{proof}

The following is an application of Theorem~\ref{T:characterisation-cov} (along with \cite[Theorem 8.7]{BLG}) which allows us to characterise all the $1$-tilting cotorsion pairs over a commutative semihereditary domain (for example, for the category of abelian groups) such that $\A$ is covering. 
\begin{expl}
Let $R$ be a commutative semihereditary ring and $(\A, \T)$ a $1$-tilting cotorsion pair in $\ModR$ with associated Gabriel topology $\G$. Then by \cite[Theorem 5.2]{H}, $\G$ is a perfect Gabriel topology. Moreover, $R/J$ is a coherent ring for every finitely generated $J \in \G$, so $R/J$ is perfect if and only if $R/J$ is artinian, \cite[Theorem 3.3 and 3.4]{Ch}. As $R/J$ is artinian, there are finitely many (finitely generated) maximal ideals and the Jacobson radical of $R/J$ is a nilpotent ideal. Therefore in this case, $\G$ has a subbasis of ideals of the form $\{\m^{k} \mid \m \in \Max R \cap \G, k \in \bbN \}$ and moreover all the maximal ideals of $R$ contained in $\G$ are finitely generated.

Moreover, if $R$ is a commutative semihereditary ring, the classical ring of quotients $Q(R)$ is Von Neumann regular. By \cite[Example XI.4.2]{Ste75}, the classical ring of quotients coincides with the maximal flat epimorphic ring of quotients $Q_\text {tot}(R)$(see \cite[Section XI.4]{Ste75}). Thus, for a $1$-tilting cotorsion pair $(\A, \T)$ as in the previous paragraph, $R \to R_\G$ is a perfect localisation (and a monomorphism) so by \cite[Proposition XI.4.1]{Ste75} $R \to Q(R)$ factors through a ring monomorphism $R_\G \to Q(R)$. It follows that if $R_\G$ is perfect then $R_\G$ coincides with its classical ring of quotients, so $R_\G = Q(R_\G)=Q(R)$. Thus if $\A$ provides covers, the $1$-tilting cotorsion pair is $(\A, Q(R)/R^\perp)$ and moreover $Q(R)$ is a semisimple ring since it is Von Neumann regular and perfect. 


In particular, in the case of $R = \bbZ$ Theorem 8.7 in \cite{BLG} implies that every $1$-tilting class $\T$ is enveloping as $\bbZ$ is semihereditary and for any proper ideal $a\bbZ$ of $\bbZ$, $\bbZ/a\bbZ$ is artinian. 

On the other hand, the only $1$-tilting cotorsion pair in $\Mod \bbZ$ that provides for covers is $(\A, \bbQ^\perp)$, that is the $1$-tilting cotorsion pair associated to the $1$-tilting module $\bbQ \oplus \bbQ/\bbZ$. 

\end{expl}

 \bibliographystyle{alpha}
\bibliography{references}
 \end{document}